\documentclass[final]{article}
\usepackage{amsmath,amssymb,dsfont}
\usepackage{graphicx,color}
\usepackage{enumitem, xspace, float}

\newcommand{\commentout}[1]{}

\newtheorem{theorem}{Theorem}[section]
\newtheorem{lemma}[theorem]{Lemma}
\newtheorem{corollary}[theorem]{Corollary}
\newtheorem{proposition}[theorem]{Proposition}
\newtheorem{definition}[theorem]{Definition}
\newtheorem{remark}[theorem]{Remark}

\newenvironment{proof}{\medskip\textit{Proof ---\;}}{\\ 
    \null\  \hfill\textbf{Q.E.D.}\medskip}
 
\def\R{\mathbb R}

\def\eps{\varepsilon}

\def\e{\varepsilon}

\newcommand{\ds}{\,\mathrm{d}s}
\newcommand{\dt}{\,\mathrm{d}t}

\def\HRL{{\displaystyle{\mathop{\scriptstyle{(y,s)\to
(x,t)}}_{\varepsilon\to 0}}}}
\def\limssup{\mathop{\rm lim\,sup\!^*\,}}
\def\limiinf{\mathop{\rm lim\,inf_*\,}}

\def\eps{\varepsilon}

\def\xe{x_\eps}
\def\te{t_\eps}

\def\X{X}

\newcommand{\Q}[1]{{Q}^{(#1)}}
\DeclareMathOperator{\dist}{\mathrm{dist}}

\newcommand{\Man}[1]{\mathbf{M}^{#1}}

\renewcommand{\d}{\mathrm{d}}

\newcommand{\hyp}[1]{$(\mathbf{H}_{#1})$}

\newcommand{\BL}{\mathbf{BL}}
\newcommand{\B}{\mathbf{B}}
\renewcommand{\L}{\mathbf{L}}

\newcommand{\cT}{\mathcal{T}}
\renewcommand{\Q}[1]{$\mathcal{Q}(#1)$}

\newcommand{\AFS}{{\rm (AFS)}\xspace}
\newcommand{\LAHF}{{\rm (LAHF)}\xspace}
\newcommand{\AHG}{{\rm (AHG)}\xspace}
\newcommand{\LCP}{\mathrm{LCP}}
\newcommand{\HJBSD}{{\rm (HJB-SD)}\xspace}
\newcommand{\RS}{{\rm (RS)}\xspace}

\newcommand{\TC}{{\bf (TC)}\xspace}
\newcommand{\NC}{{\bf (NC)}\xspace}
\newcommand{\LP}{{\bf (LP)}\xspace}

\newcommand{\TCBL}{{\bf (TC-BL)}\xspace}
\newcommand{\NCBL}{{\bf (NC-BL)}\xspace}
\newcommand{\M}{\mathbb{M}}

\newcommand{\toRS}{\mathop{\longrightarrow}\limits^{RS}}

\newenvironment{enum}{
    \begin{enumerate}[label=$(\roman *)$, topsep=0pt, leftmargin=4em, 
            itemsep=0pt]}{\end{enumerate}\medskip}

\newcommand{\distH}{\dist_\mathrm{H}}

\begin{document}

    \title{(Almost) Everything You Always Wanted to Know About Deterministic
        Control Problems in Stratified Domains} 
    
    \author{G. Barles \& E. Chasseigne
    \thanks{Laboratoire de Math\'ematiques et
            Physique Th\'eorique (UMR CNRS 7350), F\'ed\'eration Denis Poisson
            (FR CNRS 2964), Universit\'e Fran\c{c}ois Rabelais, Parc de
            Grandmont, 37200 Tours, France. Email:
            Guy.Barles@lmpt.univ-tours.fr,
            Emmanuel.Chasseigne@lmpt.univ-tours.fr . }
    \thanks{This work was partially supported by the ANR HJnet ANR-12-BS01-0008-01 and by EU under
  the 7th Framework Programme Marie Curie Initial Training Network
  ``FP7-PEOPLE-2010-ITN'', SADCO project, GA number 264735-SADCO.} 
    }

\date{}
\maketitle

\begin{abstract}
    \noindent We revisit the pioneering work of Bressan \& Hong on
    deterministic control problems in stratified domains, \textit{i.e.}  control problems
    for which the dynamic and the cost may have discontinuities on submanifolds
    of $\R^N$. By using slightly different methods, involving more partial
    differential equations arguments, we $(i)$ slightly improve the assumptions
    on the dynamic and the cost; $(ii)$~obtain a comparison result for general
    semi-continuous sub and supersolutions (without any continuity assumptions
    on the value function nor on the sub/supersolutions); $(iii)$ provide a general
    framework in which a stability result holds.  
\end{abstract}

\noindent {\bf Key-words}: Optimal control, discontinuous dynamic, Bellman
Equation, viscosity solutions.  \\
{\bf AMS Class. No}:
49L20,   
49L25,   
35F21.   

\section{Introduction}

In a well-known pioneering work, Bressan \& Hong \cite{BH} provide a rather
complete study of deterministic control problems in stratified domains, i.e.
control problems for which the dynamic and the cost may have discontinuities on
submanifolds of $\R^N$. In particular, they show that the value-function
satisfies some suitable Hamilton-Jacobi-Bellman (HJB) inequalities (in the viscosity
solutions' sense) and were able to prove that, under certain conditions, one has
a comparison result between sub and supersolutions of these HJB equations,
ensuring that the value function is the unique solution of these equations.

The aim of this article is to revisit this work by $(i)$ slightly improving the
assumptions on the dynamic and the cost, in a (slightly) more general
framework; $(ii)$ obtaining a comparison result for general semi-continuous sub
and supersolutions 
(while in \cite{BH} the subsolution has to be H\"older
continuous, and this turns into an a priori assumption on the value function that
we do not need here)
; $(iii)$ providing a general (and to our point of view,
natural) framework in which a stability result holds. 

In order to be more specific, even if we are not going to enter into details in
this introduction, the first key ingredient is the ``stratification'', namely
writing the whole space as $$\R^N=\Man{0}\cup\Man{1}\cup\dots\cup\Man{N}\; ,$$
where, for each $k=0\dots N$, $\Man{k}$ is a $k$-dimensional embedded
submanifolds of $\R^N$, the $\Man{k}$ being disjoint. The reader may consider
that the $\Man{k}$ are the subsets of $\R^N$ where the dynamic and cost have
discontinuities, which may also mean that, as in \cite{BH}, on certain
$\Man{k}$, there is a specific control problem whose dynamic and cost have
nothing to do with the dynamic and cost outside $\Man{k}$. But as in our
previous papers in collaboration with A.~Briani dedicated to co-dimension 1 type
discontinuities \cite{BBC1,BBC2}, part of the dynamic and cost on $\Man{k}$ is
some kind of ``trace'' of the dynamic and cost outside $\Man{k}$. In \cite{BH},
the regularity imposed on the $\Man{k}$ is $C^{1}$, while in our case it depends
on the controllability of the system: $C^1$ is the controllable case,
$W^{2,\infty}$ otherwise. This additional regularity may be seen as the price to
pay for having no continuity assumption on either the value function nor the
subsolutions for obtaining the equation and proving the comparison result.

The next ingredient is the control problem, \textit{i.e.} the dynamic and cost. Here we
are not going to enter at all into details but we just point out key facts.
First, contrarily to \cite{BH}, we use a general approach through differential
inclusions and we do not start from dynamic $b_k$ and cost $l_k$ defined on
$\Man{k}$. This may have the disadvantage to lead to a more difficult checking
of the assumptions in the applications but, for example, since most of our
arguments are local, the global Lipschitz assumption on the $b_k$ can be reduced
to a locally Lipschitz one. But the most interesting feature are the
controllability assumptions ---and we hope to convince the reader that they are
natural: for each $k$, we assume that the system is controllable w.r.t. the
normal direction(s) of $\Man{k}$ in a neighborhood of each $\Man{k}$, while the
dynamic and cost should also satisfy some continuity assumptions in the tangent
direction(s). This controllability assumption has a clear interpretation: if,
in a neighborhood of $\Man{k}$, the controller wants to go to $\Man{k}$, then he
can do it, and in the same way he can avoid $\Man{k}$ if this is its choice.
This avoids useless discontinuities (which are not ``seen'' by the
system) and cases where the value functions have discontinuities. We point out
that this normal controllability is a key assumption to prove that the value
function satisfies the right HJB inequalities {\em without assuming a priori
that it is continuous} but also it is a key argument in the comparison and 
stability results as this was already the case in \cite{BBC2}.

Except our slightly different approach, the viscosity sub and supersolutions
inequalities are the same as in \cite{BH}, even if the formulation coming from
the differential inclusion and the set-valued maps for the dynamic and cost
changes a litle bit the form of the Hamiltonians. The next step is more important
since it concerns the comparison of {\em any semi-continuous} sub and
supersolutions: here our proof differs from \cite{BH} since it involves more
partial differential equations (pde for short) arguments and less control ones.
A key step, already used but not in a such systematic way in \cite{BBC2}, is to
completely localize the comparison result, \textit{i.e.} to reduce to the proof of
comparison results in (small) balls. Once this is done, the assumptions on the
$\Man{k}$ allow to reduce the case when they are just affine subspaces and the
key arguments of \cite{BBC2} can be applied (regularisation in the tangent
directions to $\Man{k}$ combined by a key control-pde lemma). It is worth
pointing out anyway that, as in \cite{BH}, the proof is done by induction on the
codimension of the encountered discontinuities: local comparison in $\Man{N}$,
then successively in $\Man{N}\cup\Man{N-1}$,
$\Man{N}\cup\Man{N-1}\cup\Man{N-2}$, ..., the previous comparison result
providing the key argument for the next step. We refer to the
    beginning of Section~\ref{CUCVF} for a more explicit exposition of the
    induction argument.

Finally we provide the stability result, which extends the one proved in
\cite{BBC2} to the more complicated framework we have here but the idea remains
the same: roughly speaking, the normal controllability implies that the
half-relaxed limits on $\Man{k}$ can be computed by using only the restrictions
of the functions on $\Man{k}$, allowing to pass to the limit on the specific
inequalities on $\Man{k}$ (in particular for the subsolutions).

Recently, control problems in either discontinuous coefficients situations or in
stratified domains or even on networks have attracted more and more attention.
Of course, we start by recalling the pioneering work by Dupuis \cite{Du} who
constructs a numerical method for a calculus of variation problem with
discontinuous integrand. Problems with a discontinuous running cost were
addressed by either Garavello and Soravia \cite{GS1,GS2}, or Camilli and
Siconolfi \cite{CaSo} (even in an $L^\infty$-framework) and Soravia \cite{So}.
To the best of our knowledge, all the uniqueness results use a special structure
of the discontinuities as in \cite{DeZS,DE,GGR} or an hyperbolic approach as in
\cite{AMV,CR}. More in the spirit of optimal control problem on stratified
domains are the ones of Barnard and Wolenski \cite{BaWo} (for flows
invariances), Rao and Zidani \cite{RaZi} and Rao, Siconolfi and Zidani
\cite{RaSiZi} who proved comparison results but with more restrictive
controlability assumptions and without the stability results we can provide. For
problems on networks which partly share the same kind of difficulties, we refer
to Y. Achdou, F. Camilli, A. Cutri, N.  Tchou\cite{ACCT}, C. Imbert, R. Monneau,
and H. Zidani \cite{IMZ}, F.~Camilli and D.~Schieborn \cite{ScCa} and C. Imbert
and R. Monneau \cite{IM-ND,IM} where more and more pde methods are used, instead of
control ones. A multi-dimensional version (ramified spaces) for Eikonal type
equations is given F.~Camilli, C.  Marchi and D.~Schieborn \cite{CaMaSc} and for
more general equations in C.~Imbert and R. Monneau \cite{IM-ND}.

We end this introduction by mentioning that this paper is
    focused on the specific difficulties related to stratified domains. Hence, we
    assume that the reader is familiar with the theory of deterministic control
    problems, including the approach through differential inclusions and the connections with
    HJB equations using viscosity solutions. Good references on this subject are \cite{AF}, \cite{BCD} and
    \cite{fs}. Let us also recall that, as was said above, we derive here
    a general (theoretical) framework. In a forthcoming paper we will
    treat several specific examples and show how the present framework applies.

The article is organized as follows: in Section~\ref{CPSD}, we describe the
control problem in a full generality; this gives us the opportunity to provide
all the notations and recall well-known general results which are useful in the
sequel (in particular the results related to supersolutions properties). Then we
have to revisit the notion of stratification and we take this opportunity to
introduce the assumptions we are going to use throughout this article
(Section~\ref{AS}). Section~\ref{CHJB} contains the (subsolutions) properties
which are specific to this context. Then we address the question of the
comparison result (Section~\ref{CUCVF}), reducing first to the case of the
comparison in (small) balls which allows to flatten the submanifolds $\Man{k}$.
Section~\ref{NESR} is devoted to the stability result and we conclude the
article with a section collecting typical examples and extensions.

\

\noindent\textbf{TERMINOLOGY ---}

\begin{tabular}{ll}
\AFS & Admissible Flat Stratification\\
\HJBSD & Hamilton-Jacobi-Bellman in Stratified Domains\\ 
\AHG & Assumptions on the Hamiltonian in the General case\\
\LAHF & Local Assumptions on the Hamiltonians in the Flat case\\
\RS & Regular Stratification\\
\TC & Tangential Continuity\\
\NC & Normal Controllability\\
\LP & Lipschitz Continuity\\
$\LCP(\Omega)$ & Local Comparison Result in $\Omega$\\
$\ \,\M$ & a general regular stratification of $\R^N$
\end{tabular}

\section{Control Problems on Stratified Domains (I):\\
    Generalities or what is always true}\label{CPSD}

In this section, we consider control problems in $\R^N$ where the dynamics and
costs may be discontinuous on the collection of submanifolds $\Man{k}$ for
$k<N$. In this first part, we describe the approach using differential
inclusions and we recall all the properties of the value-function which are {\em
always true}, \textit{i.e.} results where the structure of the stratification does
not play any role. This is the case for all the supersolution type properties of
the value function. This part is essentially expository and is kept here in
order to have a self-contained article for the reader's convenience.
On the contrary, the subsolution's properties of the value function are more
specific and described in Section~\ref{CHJB}.

We first define a general control problem associated to a differential
inclusion. As we mention it above, at this stage, we do not need any particular
assumption concerning the structure of the stratification, nor on the control
sets.

\bigskip

\noindent\textsc{Dynamics and costs ---} We treat them both at the same time by
embedding the cost in the differential inclusion we solve below. We denote by
$\mathcal{P}(E)$ the set of all subsets of $E$.

\noindent\hyp{\mathbf{\BL}} We are given a set-valued maps
$\BL:\R^N\times[0,T]\to\mathcal{P}(\R^{N+1})$ satisfying
\begin{enum}
\item The map $(x,t)\mapsto\BL(x,t)$ has compact, convex images
      and is upper semi-continuous;
  \item There exists $M>0$, such that for any $x\in\R^N$ and $t> 0$,
      $$\BL(x,t)\subset \big\{(b,l)\in\R^N\times\R:|b|\leq M; | l | \leq
          M\big\}\,,$$
  \end{enum}
where $|\cdot|$ stands for the usual euclidian norm in $\R^N$ (which reduces to
the absolute value in $\R$, for the $l$-variable). If $(b,l) \in \BL(x,t)$, 
$b$ corresponds to the dynamic and $l$ to the running cost, and Assumption \hyp{\BL}-$(ii)$ means
that both the dynamics and running costs are uniformly bounded.
In the following, we sometimes have to consider separately dynamics and running costs and to do so, we set
$$\B(x,t)=\big\{b \in\R^N;\ \hbox{there exists $l\in\R$ such that  } (b,l)\in\BL(x,t) \big\}\; ,$$
and analogously for $\L(x,t)\subset\R$.

We recall what upper semi-continuity means here: 
a set-valued map $x\mapsto F(x)$ is upper-semi continuous at $x_0$ if for any
open set $\mathcal{O}\supset F(x_0)$, there exists an open set $\omega$
containing $x_0$ such that $F(\omega)\subset\mathcal{O}$. In other terms, 
$F(x)\supset \limsup_{y\to x} F(y)$.

\bigskip

\noindent\textsc{The control problem ---} as we said, we embed the accumulated cost
in the trajectory by solving a differential inclusion in $\R^N\times\R$: we look
for trajectories $(X,L)(\cdot)$ of the following inclusion
$$\frac{\d}{\dt}(X,L)(s)\in\BL\big(X(s),t-s\big)\ \text{ for a.e. }s\in[0,t)\,, 
\quad \text{and }(X,L)(0)=(x,0)\,.$$
Under \hyp{\BL}, it is well-known that given $(x,t)\in\R^N\times (0,T]$,
there exists a Lipschitz function $(X,L):[0,t]\to\R^N\times\R$ which is a
solution of this differential inclusion. To simplify, we just use the notation
$X,L$ when there is no ambiguity but we may also use the notations
$X_{x,t},L_{x,t}$ when the dependence in $x,t$ plays an important role.
If $(X,L)$ is a solution of the differential inclusion, we have for almost any
$s\in(0,t)$, $(\dot X,\dot L)(s)=(b,l)(s)$
for some $(b,l)(s)\in\BL(X(s),t-s)$. 
However, throughout the paper we prefer to write it this way 
$$
\begin{aligned}
\dot X(s) &= b\big(X(s),t-s\big)\\
\dot L(s) &= l\big(X(s),t-s\big)
\end{aligned}
$$
in order to remember that both $b$ and $l$ correspond to a specific choice in
$\BL(X(s),t-s)$.

Then, we introduce the value function
$$U(x,t)=\inf_{(\X,L)\in\cT(x,t)}\Big\{\int_0^t 
     l\big(\X(s),t-s\big)\dt+g\big(\X(t)\big)\Big\}\,,
$$
where $\cT(x,t)$ stands for all the Lipschitz trajectories $(\X,L)$ of the
differential inclusion which start at $(x,0)$ and the function $g:\R^N\to\R$
is the final cost. We assume throughout the paper that $g$ is
bounded and uniformly continuous.

A key standard result is the 
\begin{theorem}\label{DPP} {\bf (Dynamic Programming Principle)} 
    Under Assumptions \hyp{\BL}, the value-function $U$ satisfies
$$U(x,t)=\inf_{(\X,L)\in\cT(x,t)}\Big\{\int_0^\tau 
    l\big(\X(s),t-s\big)\dt+U \big(\X(\tau),t-\tau\big)\Big\}\,,$$
for any $(x,t)\in\R^N\times(0,T]$, $0<\tau < t$.
\end{theorem}

Next we introduce the ``usual'' Hamiltonian $H(x,t,p)$ for $x\in\R^N$,
$t\in[0,T]$ and $p\in \R^N$ defined~as
$$H(x,t,p)=\sup_{(b,l)\in\BL(x,t)}\big\{ -b\cdot p - l \big\}\,.$$
Using \hyp{\BL}, it is easy to prove that $H$ is upper semi-continuous (w.r.t.
all variables) and is convex and Lipschitz continuous as a function of $p$ only.

The second (classical) result is the
\begin{theorem}\label{SP} {\bf (Supersolution's Property)} 
    Under Assumptions \hyp{\BL}, the value-function $U$ is a viscosity supersolution of 
    \begin{equation}\label{eq:super.H}
        U_t + H(x,t, DU) = 0 \quad \hbox{in  } \R^N\times(0,T]\; .
    \end{equation}
\end{theorem}
In Theorem~\ref{SP}, we use the classical definition of viscosity supersolution
introduced by H.~ Ishii~\cite{Idef} for discontinuous Hamiltonians: we recall that a locally bounded function $w$ is a viscosity supersolution of (\ref{eq:super.H}) if its lower-semicontinuous envelope $w_*$ satifies
$$ (w_*)_t + H^*(x,t, Dw_*) \geq 0 \quad \hbox{in  } \R^N\times(0,T]\; ,$$
in the viscosity solutions' sense, i.e. when testing with smooth function $\phi$ at minimum points of $w_*-\phi$. Here, because of \hyp{\BL}{}, the Hamiltonian
$H$ is a locally bounded, usc function which is defined everywhere and therefore
$H^* = H$. For the sake of completeness, we recall that $w$ is a viscosity subsolution of (\ref{eq:super.H}) if its upper-semicontinuous envelope $w^*$ satifies
$$ (w^*)_t + H_*(x,t, Dw^*) \leq 0 \quad \hbox{in  } \R^N\times(0,T]\; ,$$
in the viscosity solutions' sense, i.e. when testing with smooth function $\phi$ at maximum points of $w^*-\phi$. But this definition of subsolution in $\R^N\times(0,T]$ is not the one we are going to use below.

Here and below we have chosen a formulation of viscosity solution which holds up to time
$T$, i.e. on $(0,T]$ instead of $(0,T)$, to avoid the use of terms of the form
$\eta/(T-t)$ in comparison proofs or results like \cite[Lemma~2.8, p.41]{Ba}

We also point out that both Theorem~\ref{DPP} and \ref{SP} hold in a complete general
setting, independently of the stratification we may have in mind.

We conclude this first part by a converse result showing that supersolutions
always satisfy a super-dynamic programming principle: again we remark that this
result is independent of the possible discontinuities for the dynamic or cost.

\begin{lemma}\label{lem:super.dpp}
Let $v$ be a lsc supersolution of $v_t+H(x,t,Dv)=0$ in $\R^N\times(0,T]$. Then, 
    for any $(x,t)\in\R^N\times(0,T]$ and any $0<\sigma <t$,
    \begin{equation}\label{ineq:super.dpp}
        v(x,t)\geq
        \inf_{(\X,L)\in\cT(x,t)}
        \Big\{\int_0^\sigma 
        l\big(\X(s),t-s\big)\ds+v\big(\X(\sigma),t-\sigma\big)\Big\}
    \end{equation}
\end{lemma}
\begin{proof}
For $M$ given by \hyp{\mathbf{\BL}}, we consider the sequence of Hamiltonians
    $$H_\delta(x,t,p):=\sup_{|b|\leq M, |l|\leq M} 
    \big\{ -b\cdot p -l - \delta^{-1}\psi(b,l,x,t)\big\}\,,$$
where 
$$\psi(b,l,x,t)=\inf_{(y,s) \in \R^N\times [0,T]} 
\Big(\dist\big((b,l),\BL(y,s)\big)+ |y-x|+|t-s|\Big)\; , $$
$\dist(\cdot,\BL(y,s))$ denoting here the distance here the set $\BL(y,s)$. 
Noticing that $\psi$ is Lipschitz continuous and  that $\psi(b,l,x,t)=0$ if
$(b,l)\in\BL(x,t)$, the following properties are easy to obtain

\noindent $(i)$ For any $\delta >0$, $H^\delta\geq H$ and therefore $v$ is a lsc
supersolution of $v_t+H_\delta (x,t,Dv)=0$,\\ 
$(ii)$ The Hamiltonians $H_\delta$ are (globally) Lipschitz continuous w.r.t.
all variables,
\\ $(iii)$ $H_\delta \downarrow H$ as $\delta \to 0$, all the other variables being
fixed.\\

By using $(i)$ and $(ii)$, it is clear that $v$ satisfies the Dynamic
Programming Principle for the control problem associated to $H_\delta$, namely
$$v(x,t)\geq
        \inf_{(X,L)}
        \Big\{\int_0^{t\wedge\sigma}
        l_\delta\big(X(s),t-s\big)\ds+v\big(X(t\wedge
        \sigma),t-t\wedge\sigma\big)\Big\}\; ,
$$
where $(X,L)$ solves the odes $\dot X (s)=
b(s)$, $\dot L (s)=l(s)$,
the controls $b(\cdot)$, $l(\cdot)$ satisfy $|b(s)|,|l(s)|\leq M$ and the cost is 
$$l_\delta \big(X^\delta(s),t-s\big)= l\big(s\big) +
\delta^{-1}\psi\Big(b(s),l(s),
X^\delta(s),t-s\Big)\; .$$
To conclude the proof, we have to let $\delta$ tend to $0$. To do so, we pick an optimal or $\delta$-optimal trajectory, i.e. $(X^\delta,L^\delta)$ such that
$$        \inf_{(X,L)}
        \Big\{\int_0^{t\wedge\sigma}
        l_\delta\big(X(s),t-s\big)\ds+v\big(X(t\wedge
        \sigma),t-t\wedge\sigma\big)\Big\} \geq 
     \int_0^{t\wedge\sigma}
        l_\delta\big(X^\delta (s),t-s\big)\ds+v\big(X^\delta(t\wedge
        \sigma),t-t\wedge\sigma\big) -\delta \; .$$
Since $\dot X^\delta =b^\delta, \dot L^\delta=l^\delta$ are uniformly bounded,
standard compactness arguments imply that up to the extraction of a
subsequence, we may assume that  $X^\delta, L^\delta$ converges uniformly on
$[0,t\wedge\sigma]$ to $(X,L)$. And we may also assume that they derivatives
converge in $L^\infty$ weak-* (in particular $\dot L^\delta =l^\delta$).

We use the above property of $X^\delta, L^\delta$, namely
\begin{equation}\label{spdd}
\int_0^{t\wedge\sigma}
        l_\delta\big(X^\delta (s),t-s\big)\ds+v\big(X^\delta(t\wedge
        \sigma),t-t\wedge\sigma\big) -\delta \leq v(x,t)\; ,
\end{equation}
in two ways: first by multiplying by $\delta$, we get
$$
\int_0^{t\wedge\sigma} \psi\Big(b^\delta (s),l^\delta (s),
X^\delta(s),t-s\Big)ds = O(\delta)\; .$$
But $\psi$ is convex in $(b,l)$ since the $\BL(y,s)$ are convex and if $(b^\delta ,l^\delta)$ converges weakly to $(b,l)$ (and $X^\delta$ converges uniformly), we have
$$\int_0^{t\wedge\sigma} \psi\Big(b(s),l(s),
X(s),t-s\Big)ds \leq \liminf_\delta \int_0^{t\wedge\sigma} \psi\Big(b^\delta (s),l^\delta (s),
X^\delta(s),t-s\Big)ds = 0\; .$$
Finally we remark that $\psi \geq 0$ and $\psi(b,l,x,t)=0$ if and only if $(b,l)\in \BL(x,t)$, therefore $(X,L)$ is a solution of the $\BL$-differential inclusion.

In order to conclude, we come back to (\ref{spdd}) and we remark that $l_\delta\big(X^\delta (s),t-s\big) \geq l^\delta (s)$ since $\psi\geq 0$. Therefore
$$\int_0^{t\wedge\sigma}
        l^\delta (s) \ds+v\big(X^\delta(t\wedge
        \sigma),t-t\wedge\sigma\big) -\delta \leq v(x,t)\; ,$$
and we pass to the limit in this inequality using the lower-semicontinuity of
$v$, together with the weak convergence of $l^\delta$ and the uniform
convergence of $X^\delta$. This yields 
$$\int_0^{t\wedge\sigma}
        l (s) \ds+v\big(X(t\wedge
        \sigma),t-t\wedge\sigma\big) \leq v(x,t)\; ,$$
and recalling that $(X,L)$ is a solution of the $\BL$-differential inclusion and
taking the infimum in the left-hand side over all solution of this differential
inclusion gives the desired answer.
\end{proof}

\section{Admissible Stratifications:\\
    how to re-read Bressan \& Hong Assumptions?}\label{AS}

In this section, we define the notion of \textit{Admissible Stratification},
which specifies the structure of the discontinuity set of $\BL$ as was
considered in \cite{BH}. We point out that, besides of the precise regularity we
will impose in connection with the control problem, this notion is nothing but
the notion of Whitney Stratification, based on the Whitney condition
\cite{W1,W2}, see below Lemma~\ref{flat-lem} and Remark~\ref{rem:whitney}.
We first do it in the case of a flat stratification;
the non-flat case is reduced to the flat one by suitable local charts.

\subsection{Admissible Flat Statification}
\label{sect:afs}

We consider here the stratification introduced in Bressan and Hong
\cite{BH} but in the case when the different embedded submanifolds of $\R^N$ are
locally affine subspace of $\R^N$. More precisely
$$\R^N=\Man{0}\cup\Man{1}\cup\cdots\cup\Man{N}\; ,$$
where the $\Man{k}$ ($k=0..N$) are disjoint submanifolds of $\R^N$. We
say that $\M=(\Man{k})_{k=0..N}$ is an \textit{Admissible Flat Stratification}
\AFS, the following set of hypotheses is satisfied

\bigskip

\noindent\AFS-$(i)$ For any $x \in \Man{k}$, there exists $r>0$ and $V_k$ a
 $k$-dimensional linear subspace of $\R^N$ such that  $$ B(x,r) \cap
         \Man{k} = B(x,r) \cap (x+V_k)\; .$$ Moreover $B(x,r) \cap
         \Man{l}=\emptyset$ if $l<k$.

\noindent \AFS-$(ii)$ If $\Man{k} \cap \overline {\Man{l}}\neq \emptyset$ for some $l>k$
       then $\Man{k} \subset \overline {\Man{l}}$.

\noindent \AFS-$(iii)$ We have $\overline{\Man{k}} \subset
       \Man{0}\cup\Man{1}\cup\cdots\cup\Man{k}$.

\begin{remark}Condition \AFS-$(i)$ implies that the set $\Man{0}$, if not
           void, consists of isolated points. Indeed, in the case $k=0$,
           $V_k=\{0\}$.
\end{remark}

We point out that these assumptions are equivalent (for the flat case)  to the
assumptions of Bressan \& Hong \cite{BH}. Indeed, we both assume a decomposition
such that the submanifolds are disjoints and the union of all of them coincide
with $\R^N$ but in order to describe the allowed stratifications we define in
a different way the submanifolds $\Man{k}$. The key point is that for us
$\Man{k}$ is here a $\mathbf{k}$-dimensional submanifold while, in \cite{BH},
the $\Man{j}$ can be of any dimension.  In other words, {\em our}  $\Man{k}$ is
the union of all submanifolds of dimension $k$ in the stratification of Bressan
\& Hong.

With this in mind it is easier to see that our assumptions  \AFS-$(ii)$-$(iii)$
are equivalent to the following assumption of Bressan and Hong: if $\Man{k}
\cap \overline {\Man{l}} \neq \emptyset$  then $\Man{k} \subset
\overline{\Man{l}}$  for all indices  $l,k$ without asking $l>k$ in our case.
But according to the last part of \AFS-$(i)$, $\Man{k} \cap \overline
{\Man{l}} = \emptyset$ if $l<k$: indeed for any $x\in \Man{k}$, there
exists $r>0$ such that $B(x,r) \cap \Man{l}=\emptyset$. This property
clearly implies \AFS-$(iii)$.

In order to be more clear let us consider a stratification in $\R^3$ induced by
the upper half-plane $\{x_3>0,x_2=0\}$ and the $x_2$-axis (see
figure 1. below).

\begin{figure}[H]\label{fig.0} 
    \begin{center}
   \includegraphics[width=8cm]{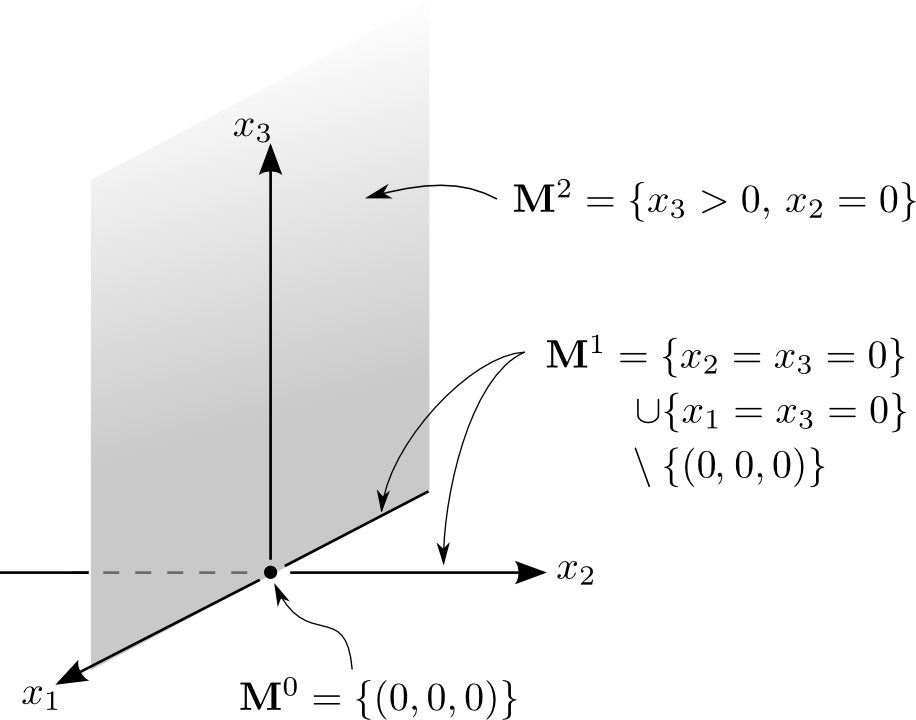}
   \caption{Example of a 3-D stratification}
   \end{center}
\end{figure}

The stratification we use in this case requires first to set
$\Man{2}=\{x_3>0,x_2=0\}$. The boundary of $\Man{2}$ which is the $x_1$-axis is
included in $\Man{1}\cup\Man{0}$ and of course, we have to set here
$\Man{0}=\{(0,0,0)\}$. Thus, $\Man{1}$ consists of four connected components
which are induced by the $x_1$- and $x_2$-axis (but excluding the origin, which
is in $\Man{0}$). Notice that in this situation, the $x_3$-axis has no
particular status, it is included in $\Man{2}$.

On the other hand, notice that \AFS-$(ii)$ FORBIDS the following decomposition
of $\R^3$
$$ \Man{2} =\{x_3>0\ ,\ x_2=0\}\;,\; \Man{1} =\{x_1=x_3=0\} \cup \{
    x_2=x_3=0\}\;,\; \Man{3}=\R^3-\Man{2}-\Man{1}\; ,$$
because $(0,0,0)\in \Man{1} \cap\, \overline {\Man{2}}$ but clearly $\Man{1}$ is not
included in $\overline {\Man{2}}$.

As a consequence of this definition we have following result which will be
usefull in a tangential regularization procedure (see Figure 2 below)
\begin{lemma}\label{flat-lem} Let $\M=(\Man{k})$ be an \AFS of
    $\R^N$, let $x$ be in $\Man{k}$ and $r, V_k$ as in \AFS-$(i)$ and
    $l>k$. Then there exists $r' \leq r$ such that, if $B(x,r') \cap \Man{l} \neq
    \emptyset$, then for any $y \in B(x,r') \cap \Man{l}$, $B(x,r') \cap (y+V_k)
    \subset B(x,r') \cap \Man{l}$.  
\end{lemma}

\begin{proof} We first consider the case when $l=k+1$. We argue by contradiction
    assuming that there exists $z \in B(x,r') \cap (y+V_k)$, $z\notin
    \Man{k+1}$. We consider the segment $[y,z]=\{ty+(1-t)z,\ t \in [0,1]\}$.
    There exists $t_0 \in [0,1]$ such that $x_0:= t_0 y+(1-t_0)z \in
    \overline{\Man{k+1}}-\Man{k+1}$. But because of the \AFS conditions,
    $\overline{\Man{k+1}}-\Man{k+1} \subset \Man{k}$ since no point of
    $\Man{0},\Man{1},\cdots \Man{k-1}$ can be in the ball. Therefore $x_0$
    belongs to some $\Man{k}$, a contradiction since $B(x,r) \cap \Man{k} =
    B(x,r) \cap (x+V_k)$ which would imply that $y \in \Man{k}$.

    For $l>k+1$, we argue by induction. If we have the result for $l$, then we
    use the same proof as above if $y \in \Man{l+1}$: there exists $z \in
    B(x,r') \cap (y+V_k)$, $z\notin \Man{l+1}$ and we build in a similar way
    $x_0 \in \overline{\Man{l+1}}-\Man{l+1}=\Man{l}$. But this is again a
    contradiction with the fact that the result holds for $l$; indeed $x_0 \in
    \Man{l}$ and $y \in x_0 + V_k \in \Man{l+1}$.  
\end{proof}

\begin{figure}[H]\label{fig.1} 
    \begin{center}
   \includegraphics[width=8cm]{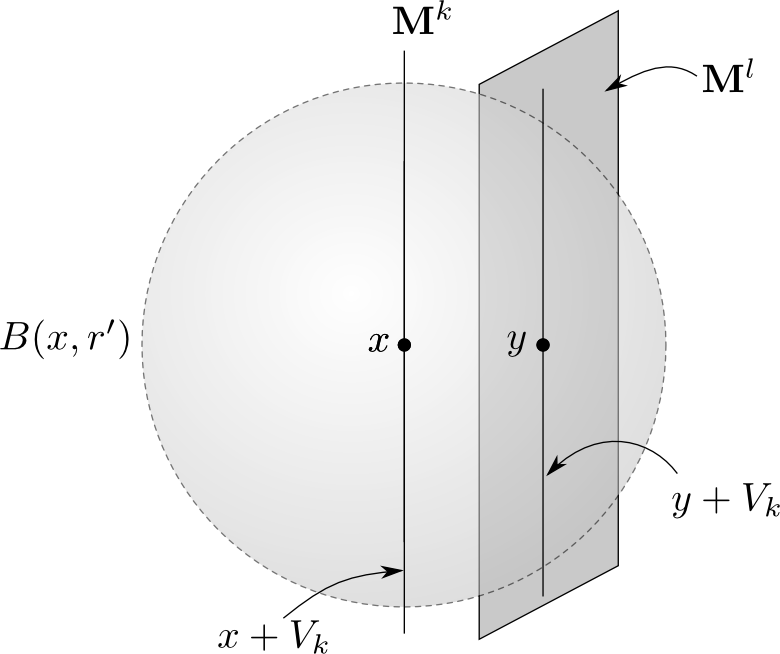}
   \caption{local situation}
   \end{center}
\end{figure}

    \begin{remark}\label{rem:whitney}
        In this flat situation, the tangent space at $x$ is $T_x:=x+V_k$
        while the tangent space at $y$ is $T_y:=y+V_l$, where $l>k$. The previous
        lemma implies that if $(y_n)_n$ is a sequence converging to $x$, then the
        limit tangent plane of the $T_{y_n}$ is $x+V_l$ and it contains $T_x$, which is exactly the
        Whitney condition ---see {\rm \cite{W1,W2}}.
    \end{remark}

\subsection{General Regular Stratification}

\begin{definition}\label{def:RS} We say that $\M=(\Man{k})_{k=0..N}$ 
    is a general regular stratification \RS of $\R^N$ if\\[2mm]
    $(i)$ the following decomposition holds:
    $\R^N=\Man{0}\cup\Man{1}\cup\cdots\cup\Man{N}$;\\[2mm]
    $(ii)$ for any $x \in \R^N$, there exists $r=r(x)>0$ and a
    $C^{1,1}$-change of coordinates $\Psi^x: B(x,r) \to \R^N$ such that the
    $\Psi^x(\Man{k}\cap B(x,r))$ form an \AFS in $\Psi^x(B(x,r))$.
\end{definition}

\begin{remark}
    If we need to be more specific, we also say that $(\M,\Psi)$ is a
    stratification of $\R^N$, keeping the reference $\Psi$ for the collection of
    changes of variable $(\Psi^x)_x$. This will be usefull in
    Section~\ref{NESR} when we consider sequences of stratifications.
\end{remark}

\noindent\textbf{Notations ---} 
The definition of regular stratifications (flat or not) allows to define, 
    for each $x\in\Man{k}$, the tangent space to $\Man{k}$ at $x$, denoted by
    $T_x\Man{k}$, which can be identified to $\R^{k}$.
    Then, if $x \in \Man{k}$ and if $r>0$ and $V_k$ are as in \AFS-$(i)$, we can
    decompose $\R^N = V_k \oplus V_k^\bot$, where $V_k^\bot$ is the orthogonal space to
    $V_k$ and for any $p \in \R^N$ we have $p= p_\top + p_\bot$ with $p_\top \in
    V_k$ and $p_{\bot}\in V_k^\bot$. In the special case
        $x\in\Man{0}$, we have $V_0=\{0\}$, $p=p_\bot$ and $T_x\Man{0}=\{0\}$.

At this stage, it remains to connect the stratification with the set-valued map
$\BL$. To do so, we first recall that the set function $\BL$ is said to be continuous
at $(x,t)\in \R^N\times\R_+$ if $\distH(\BL(y,s),\BL(x,t)) \to 0$ when $(y,s) \to (x,t)$,
where $\distH(\cdot,\cdot)$ stands for the Haussdorf distance between sets.
Now, given a regular stratification $\M=(\Man{k})_{k=0..N}$ of $\R^N$, let us
denote by $\BL|_k$ the restriction of $\BL$ to $\Man{k}\times[0,T]$
\begin{align*}\BL|_k: \Man{k}\times[0,T] & \to \mathcal{P}(\R^{N+1})\\
    (x,t) &\mapsto \BL(x,t)\cap (T_x\Man{k}\times\R)
\end{align*}

\begin{definition}
We say that the regular stratification $\M$ of $\R^N$ is
adapted to $\BL$ if for any $k\in\{0,...,N\}$, 
the restriction $\BL|_k$ is continuous on $\Man{k}\times[0,T]$.
In particular, the set of discontinuities of the restriction of
$\BL$ to any $\overline{\Man{k}}\times[0,T]$ is 
$(\Man{0}\cup\Man{1}\cup\cdots\cup\Man{k-1})\times[0,T]$.
\end{definition}

\subsection{Hamiltonians}

Considering a regular stratification $\M$ adapted to $\BL$, we introduce the
associated Hamiltonians: if $x\in \Man{k}$, $t\in [0,T]$ and $p \in T_x\Man{k}$,
the tangential Hamiltonian on the $\Man{k}$-submanifold is defined by
\begin{equation}\label{def:Hk}
H^k(x,t,p):=\sup_{\substack{(b,l)\in\BL(x,t)\\ b\in T_x\Man{k}}}\big\{ -
    b\cdot p - l\big\}\,.
\end{equation}
The continuity requirements on the maps $\BL|_k$ (see above) together with the
compactness of each $\BL(x,t)$ implies the continuity of $H^k$ in $(x,t,p)$, for
any $k$. In this definition (where we have implicitly identified $T_x\Man{k}$ as
a subspace of $\R^N$), it is clear that $H^k$ depends on $p$ only through its
projection on $T_x\Man{k}$ but we keep the notation $p$ to simplify the
notations.

Notice that in the special case $k=0$, since $T_x\Man{0}=\{0\}$ the Hamiltonian
reduces to:
$$H^0(x,t)=\sup_{\substack{(b,l)\in\BL(x,t)\\ b=0}}\big\{-l\big\}=
-\inf \big\{l: (0,l)\in\BL(x,t)\big\}\,.$$

In order to prove comparison for the complemented problem, we need some 
assumptions on the Hamiltonians that we formulate first in the case of an \AFS.

For any $x\in\R^N$, if $x \in \Man{k}$ and $r=r(x)$ is given by \AFS-$(i)$, there
exist three constants $C_i=C_i(x,r)$ ($i=1..3$) and a modulus of continuity
$m:[0,+\infty) \to [0,+\infty) $ with $m(0+)=0$  such that

\noindent \TC \textit{Tangential Continuity}: for any $0\leq k\leq j\leq N$, for
any $t,t'\in [0,T]$, if $y_1,y_2\in\Man{j}\cap B(x,r)$ with $y_1-y_2 \in V_k$,
then $$|H^j(y_1,t,p)-H^j(y_2,t',p)|\leq C_1\big\{|y_1-y_2|+m(|t-t'|)\big\}|p|
+m\big(|y_1-y_2|+|t-t'|\big)\,.$$
\ 
We point out the importance of Lemma~\ref{flat-lem} which implies that this is
actually an assumption on any $y_1$ (or $y_2$) of $\Man{j}$.\\

\noindent \NC \textit{Normal Controllability}: for any $0\leq k< j\leq N$, for
any $t\in [0,T]$, if $y\in\Man{j}\cap B(x,r)$ then $$H^j(y,t,p)\geq \delta
|p_\bot| - C_2(1+|p_\top|)\,.$$

\noindent In particular, in the special case $k=0$, we have $p=p_\bot$. So, \NC
implies the coercivity w.r.t $p$ of all the Hamiltonians $H^k$, $k=1..N$, in a
neighborhood of any point $x\in\Man{0}$ (recall that such points are isolated).

\ 

\noindent \LP \textit{Lipschitz continuity}: because of the
boundedness of $\BL$, there exists $C_3$ such that, for any $0\leq k\leq j\leq
N$, if $y\in\Man{j}\cap B(x,r)$ then 
$$|H^j(y,u,p)-H^j(y,u,q)|\leq C_3|p-q|\,.$$

It is worth pointing out that these assumptions (except perhaps \LP) are local
assumptions since they have to hold in a neighborrhood of each point in $\R^N$
and the different constants or modulus of continuity may depend on the
considered point. The strategy of proof for the comparison result will explain
this unusual feature and in particular Lemma~\ref{FLTG}.

\begin{definition}\label{def:AHG} Let $\M$ be a general regular stratification
    associated to $\BL$ and $(H^k)_{k=0..N}$ be the associated Hamiltonians.\\[2mm]
        $(i)$ In the case of an admissible flat stratification, we
        say that the associated Hamiltonians $(H^k)_{k=0..N}$ satisfy the
        \textit{Local Assumptions on the Hamiltonians in the Flat case} \LAHF if
        \TC, \NC and \LP are satisfied.\\[2mm]
    $(ii)$ In the general case, we say that the associated Hamiltonians satisfy
    the \textit{Assumption on the Hamiltonians in the general case} \AHG if the
    Hamiltonians $\tilde H^k (y,t,q):= H^k (\chi(y),t,\chi'(y)q)$ 
    satisfy the \LAHF, where $\chi=(\Psi^x)^{-1}$.
\end{definition}

In order to be complete, we give below sufficient conditions in terms of $\BL$
for \TC \& \NC to hold: the first one concerns regularity and the second one
ensures the normal coercivity of the Hamiltonians. \\

\noindent\TCBL For any $0\leq k\leq j\leq N$, for any $t\in [0,T]$, if
$y_1,y_2\in\Man{j}\cap B(x,r)$ with $y_1-y_2 \in V_k$,
$$\begin{cases}\distH\big(\B(y_1,t),\B(y_2,t)\big) \leq C_1|y_1-y_2|\,,\\[2mm] 
    \distH\big(\BL(y_1,t),\BL(y_2,t')\big) \leq
    m\big(|y_1-y_2|+|t-t'|\big)\,.\end{cases}$$

\noindent\NCBL There exists $\delta >0$ such that, for any $0\leq k< N$, for
any $t\in [0,T]$, if $y\in B(x,r)\setminus\Man{k}$ there holds
$$B(0,\delta) \cap V_k^\bot \subset \B(y,t) \cap V_k^\bot\,.$$

\noindent Here also, the case $k=0$ is particular: we impose a complete
controllability of the system in a neighborhood of $x\in\Man{0}$ since the
condition reduces to $B(0,\delta)\subset \B(y,t)$ because $V_k^\bot = \R^N$.

\noindent This normal controllability assumption plays a key role in all our
analysis: first, in the proof of Theorem~\ref{SubP} below, to obtain the
viscosity subsolution inequalities for the value function, in the comparison
proof to allow the regularization (in a suitable sense) of the subsolutions and,
last but not least, for the stability result.

\section{Control Problems on Stratified Domains (II):\\
    Subsolutions and Complemented Hamilton-Jacobi-Bellman Equations}\label{CHJB}

For the subsolution's property of $U$, the behaviour of the dynamic is going to
play a key role and we have to strengthen Assumption \hyp{\BL} by adding
continuity and controllability assumptions, \TCBL \& \NCBL which are equivalent
to \TC \& \NC. The main consequences of \TCBL \& \NCBL is the 
\begin{theorem}\label{SubP} {\bf (Subsolution's Property)}
    Under Assumptions
    \hyp{\BL}, \TCBL and \NCBL, the value-function $U$ satisfies\\[2mm]
        $(i)$ For any $k=0..(N-1)$, $U^* = (U|_{\Man{k}})^*$ on
        $\Man{k}$\,;\\[2mm] 
    $(ii)$ for any $k=0..(N-1)$, $U$ is a subsolution of 
    $$U_t+H^k(x,t,DU)=0\quad\text{on }\Man{k}\times(0,T)\,.$$
\end{theorem}

In this result, we point out -- even if it is obvious-- that $(ii)$ is a
viscosity inequality for an equation {\em restricted to $\Man{k}$}, namely it means that if $\phi$ is a smooth function on $\Man{k}\times(0,T)$ (or equivalently on $\R^N\times(0,T)$ by extension) and if $(x,t)\in \Man{k}\times(0,T)$ is a local maximum point of $U^*-\phi$ on $\Man{k}\times(0,T)$, then
$$
\phi_t(x,t) +H^k(x,t,D\phi(x,t))\leq 0\ \footnote{\label{fnt:ext} For the sake of simplicity, we have still denoted by $\phi$ the smooth extension of $\phi$ to $\R^N\times(0,T)$ and by $D\phi$ its gradient in $\R^N$ but because of the form of $H^k$, clearly only the part of $D\phi$ which is on the tangent space of $\Man{k}$ at $x$ plays a role in this inequality.}\; .
$$ 
This is
why point $(i)$ is an important fact since it allows to restrict everything
(including the computation of the usc envelope of $U$) to $\Man{k}$.

\begin{proof} We provide the proof in the case of an \AFS, the general case
    resulting from a simple change of variable.

We consider $x\in \Man{k}, t\in (0,T]$ and a sequence $(\xe,\te)\to (x,t)$
such that $$ U^*(x,t) = \lim_\e U(\xe,\te)\; .$$ We have to show that we
can assume that $\xe \in \Man{k}$. 

We assume that, on the contrary, $\xe \notin \Man{k}$ and we show how to build a sequence
of points $(\bar x_\e,\bar t_\e)_\e$ with  $\bar x_\e \in \Man{k}$ for any $k$ and with
$ U^*(x,t) = \lim_\e U(\bar x_\e,\bar t_\e)$.

By Theorem~\ref{DPP}, we have
$$U (\xe,\te) \leq \int_0^\tau 
    l\big(\X(s),t-s\big)\dt+U \big(\X(\tau),t-\tau\big)\,,$$
for any solution $(X,L)$ of the differential inclusion starting from $(\xe,0)$.
Let $\tilde x_\e$ be the projection of $\xe$ on $\Man{k}$; we have $\tilde
x_\e-\xe \in V_k^\bot$ and by \NCBL, there exists $b\in\B(y,s)$ for any $y\in
B(x,r)$ (the ball given by \AFS-$(i)$), such that $b_\bot:=\delta/2.(\tilde
x_\e-\xe)|\tilde x_\e-\xe|^{-1}$.
 
Choosing such a dynamic $b$ (with any constant cost $l$), it is clear that $X(s)\in B(x,r)$ for $s$ small enough (independent of $\e$) and for $s_\e=2|\tilde
x_\e-\xe|/\delta$, we have $\bar x_\e = X(s_\e)=\tilde x_e + y_\e$ where $y_\e \in V_k$, $|y_\e| = O(|\tilde x_\e-\xe|)$. Therefore $\bar x_\e \in \Man{k}$ by Lemma~\ref{flat-lem} and if we set $\bar t_\e = t_\e-s_\e$, we have
$$U (\xe,\te) \leq \int_0^{s_\e} l \dt+U
\big(\X(s_\e),t_\e-s_\e\big)= s_\e\, l+U \big(\bar x_\e,\bar t_\e \big)\,.$$
Finally since $s_\e \to 0$ as $\e \to 0$, we deduce that
$$ \limsup_\e U \big(\bar x_\e,\bar t_\e\big) \geq \limsup_\e
\, U (\xe,\te)=U^*(x,t)\; ,$$ which shows $(i)$ since $\bar x_\e \in
\Man{k}$.

To prove $(ii)$, we assume now that $\xe \in \Man{k}$ and we use again
Theorem~\ref{DPP} which implies 
$$U (\xe,\te) \leq \int_0^\tau 
    l\big(\X(s),t-s\big)\dt+U \big(\X(\tau),t-\tau\big)\,,$$
for any solution $(X,L)$ of the differential inclusion starting from
$(\xe,0)$. Using the continuity of $\BL|_k$, if $(b,l)$ is in the interior of $\BL|_k (x,t)$, the trajectory $X(s)$, starting from $\xe$ at time $\te$ remains on
$\Man{k}$ for $s\in[0,\tau]$ if $\tau$ is small enough (but
independent of $\e$). Thus, the viscosity inequality can
be obtained as in the standard case and we obtain the inequality for $(b,l)$ is in the whole $\BL|_k (x,t)$ by a simple passage to the limit.
\end{proof}

\
The sub and supersolution properties of the value function naturally leads us to
the following definition.
\begin{definition}\label{def:HJBSD}
    Let $\M$ be a regular stratification of $\R^N$ associated to a set-valued
    map $\BL$. 

    \noindent $\noindent (i)$ A bounded usc function $u:\R^N\times[0,T]\to\R$ is
    a viscosity subsolution of the Hamilton-Jacobi-Bellman in Stratified Domain
    {\rm [\HJBSD} for short{\rm ]}, if and only if it is a subsolution of 
    $$u_t + H^k (x,t, Du) = 0 \quad \hbox{on  } \Man{k}\times(0,T]\; ,$$
    for any $k=0..N$, i.e if, for any test-function $\phi \in
    C^1(\Man{k} \times[0,T])$ and for any local maximum point $(x,t) \in \Man{k}
    \times(0,T)$ of $u-\phi$ on $\Man{k} \times(0,T]$, we have $$\phi_t (x,t) +
    H^k (x,t, D\phi (x,t)) \leq  0 \; .$$

    \noindent $\noindent (ii)$ A bounded lsc function $v:\R^N\to\R$ is a viscosity
    supersolution of \HJBSD if it is a viscosity supersolution of
    $$v_t+H(x,t,Dv)=0  \quad \hbox{in  } \R^N \times(0,T]\; .$$
\end{definition}

    The same remark as above applies, see footnote~\ref{fnt:ext}: the extension
    of $\phi$ to all $\R^N\times(0,T)$ is still denoted by $\phi$, for the sake of
    simplicity of notations.

In the sequel, we also say that a function is a subsolution or a supersolution
of \HJBSD in a domain $D\subset \R^N \times (0,T]$ if the above properties hold
true either in $\Man{k} \times(0,T] \cap D$ or in $D$. We also say that $u$ is a
{\em strict} subsolution of \HJBSD in a domain $D\subset \R^N \times (0,T]$ if
the inequality $\leq 0$ is replaced by $\leq -\eta$ for some $\eta>0$.

\begin{remark} As in \cite{BBC1,BBC2}, we notice that additional subsolution conditions involving the tangential Hamiltonians $(H^k)_{k=0..N}$ are required on the manifolds $\Man{k}$'s . It might be surprising anyway that we have no subsolution condition related to trajectories which are leaving $\Man{k}$ for $k<N$. In fact, even if we are not going to enter into details here, these conditions can be deduced from the inequalities on $\Man{l}$ for $l>k$ in the spirit of \cite[Theorem 3.1]{BBC1}.
\end{remark}

\section{Comparison, Uniqueness and Continuity of the Value-Function}\label{CUCVF}

In this section, we provide our main comparison result for \HJBSD. Since the proof relies on proving comparison properties in different subdomains of $\R^N$, we introduce the following definition.
\begin{definition} We have a comparison result for \HJBSD in $Q=\Omega \times
    (t_1,t_2)$, where $\Omega$ is an open subset of $\R^N$ and $0\leq
    t_1<t_2\leq T$, if, for any bounded usc subsolution $u$ of \HJBSD in $Q$ and
    any bounded lsc supersolution $v$ of \HJBSD in $Q$, then
$$\| ( u- v)_+\|_{L^\infty(\overline Q)}\leq
    \| (u- v)_+\|_{L^\infty(\partial_p Q)}\,,$$
where $\partial_p Q$ denotes the parabolic boundary of $Q$, i.e. $\partial_p
Q:=\partial \Omega \times [t_1,t_2] \cup \overline \Omega \times \{t_1\}$.
\end{definition}

Our main result is the
\begin{theorem} \label{thm:main}
    We have a comparison result for \HJBSD in any subdomain $Q=\Omega \times
    (t_1,t_2)$ of $\R^N \times (0,T)$.  
\end{theorem}

In order to guide the reader in the long and unusual proof (despite it has some
common features with the global strategy in Bressan \& Hong \cite{BH} and uses
locally the ideas of \cite{BBC2}), we indicate the main steps.
\label{page:recurrence}

\begin{itemize}
\item We first show that, instead of proving a ``global'' comparison result, we
      can reduce to comparison results in ``small'' balls. Essentially this
      first step allows us to reduce to the case of ``flat stratifications'',
      namely \AFS.
\item Then we argue by induction on the dimension of the submanifolds which are
      contained in the small ball: if the small ball is included in $\Man{N}$,
      this means that there is no no discontinuities and we have a standard comparison
      result. The next step consists in proving a comparison result in the
      case when the ball intersects both $\Man{N}$ and $\Man{N-1}$, which is
      actually already done in \cite{BBC2}. Therefore the induction consists in
      proving that if we have a comparison result for any ball intersecting
      (possibly) $\Man{N},\dots, \Man{k+1}$, then it is also true for any ball
      intersecting $\Man{N},\dots,\Man{k+1},\Man{k}$.
\item To perform the proof of this result, we use three key ingredients: for
      the subsolution, the regularization by sup-convolution and then by usual
      convolution in the tangent direction to $\Man{k}$ (and this is where the
      \AFS structure is playing a key role, see Lemma~\ref{flat-lem}) together with
      the fact that a comparison result in $\Man{k+1}\cup \Man{k+2} \cdots \cup
      \Man{N}$ implies that subsolutions satisfy a sub-optimality principle in
      this domain. On the other hand, for the supersolution, the DPP allows us
      to prove an analogous ``magic lemma'' as in \cite{BBC1,BBC2}.
\end{itemize}

In order to formulate the induction, let us introduce the following statement,
where $k\in\{0,...,N\}$
\begin{itemize} 
\item[\Q{k}:]
    {\it For any ball $B\subset \Man{k}\cup\Man{k+1}\cup\dots\cup\Man{N}$,
        for any $0\leq t_1 <t_2 \leq T$ and for
        any strict subsolution $u$ of \HJBSD in $B\times (t_1,t_2]$ and any
        supersolution $v$ of \HJBSD in $B\times (t_1,t_2]$, $u-v$ cannot
        have a local maximum point in $B\times (t_1,t_2]$.} 
\end{itemize} 

    \begin{remark} We use a localized formulation of Property \Q{k} in any
        ball because we apply it below to functions which, at level $k$, are only
        subsolutions in such specific balls. 
    \end{remark}

\subsection{From local to global comparison}

Our first result consists in showing that we can
reduce the global comparison result in $\R^N\times [0,T]$ to ``local''
comparison results. Let us introduce the following version of the Local
Comparison Principle in a cylinder $\Omega\times[0,T]\subset\R^N\times[0,T]$

\noindent $\LCP(\Omega)$: \textit{for any $(x,t) \in \Omega \times (0,T]$, 
    there exists $\bar r,\bar h >0$ such that $B_{\bar r}(x) \subset \Omega$,
    $\bar h \leq t$ and one has a comparaison result in $B(x,r) \times (t-h,t)$
    for any $r\leq \bar r$ and $h\leq \bar h$.  }

\begin{lemma}\label{FLTG}
    Assume \hyp{\BL}. We have a comparaison result in $Q:=\Omega \times (0,T]$ if
    and only if $\LCP(\Omega)$ holds true. 
\end{lemma}

\begin{proof} Let $u,v$ be respectively a bounded usc subsolution $u$ and a
    bounded lsc supersolution $v$ of \HJBSD in $Q$. We consider
    $M=\sup_{\bar Q} (u-v)$. If $M\leq 0$ then we have
    nothing to prove, hence we may assume that $M>0$.

In order to replace the ``$\sup$'' by a ``$\max$'' if $\Omega$ is unbounded, we argue as in
\cite{BBC1,BBC2} and we replace $u$ by 
$$ u_\alpha (x,t):= u(x,t) -\alpha (Ct +(1+|x|^2)^{1/2}\; ,$$ 
for $0\leq \alpha \ll 1$. Proving the comparison inequality for $u_\alpha$
instead of $u$ provides the result by letting $\alpha$ tend to $0$.

With this argument, we can consider $M_\alpha=\max_{\bar Q}(u_\alpha-v)$
and we denote by $(x,t)$ a maximum point of $u_\alpha-v$. In
addition, we may assume that $t$ is the minimal time for which there exists such
a maximum point. If $t=0$ or $x \in \partial \Omega$ then the result is proved,
hence we may also assume that $x\in \Omega$ and $t>0$.

Using the assumption, we know that there exists $\bar r,\bar h >0$ such that
such that $B_{\bar r}(x) \subset \Omega$, $\bar h \leq t$ and one has a
comparaison result in $B(x,r) \times (t-h,t)$ for any $r\leq \bar r$ and $h\leq
\bar h$.

Thus, in $Q_{r,h}:=B(x,r) \times (t-h,t)$ (where $r$ and $h$ will be chosen later), we
change $u_\alpha (y,s)$ into
$$
u_{\alpha,\beta} (y,s):=u_\alpha (y,s)-\beta(\bar
C(s-t)+(|y-x|^2+1)^{1/2}-1)\; ,
$$
where $0<\beta\ll 1$. If $\bar C$ is large
enough, $u_{\alpha,\beta}$ is still a subsolution in $Q_{r,h}$ and as a
consequence of the comparison property, we have 
$$ u_{\alpha}(x,t)-v(x,t) \leq \max_{\partial_p Q_{r,h}}(u_\alpha (y,s)-\beta(\bar
(s-t)+(|y-x|^2+1)^{1/2}-1)-v(y,s))\; .$$ 
But if $y\in \partial B(x,r)$
$$
\beta(\bar C(s-t)+(|y-x|^2+1)^{1/2}-1) = \beta(\bar C(s-t)+(r^2+1)^{1/2}-1)\geq
\beta(-\bar Ch+(r^2+1)^{1/2}-1)\; ,
$$
and, since $(r^2+1)^{1/2}-1>0$, if we choose (and fix) $h$ small enough, we have
$\beta(-\bar Ch+(r^2+1)^{1/2}-1)>0$. Therefore, for such $h$,
$$
\max_{\partial B(x,r) \times [t-h,t]} ( u_\alpha (y,s)-\beta(\bar
C(s-t)+(|y-x|^2+1)^{1/2}-1)-v(y,s))<M_\alpha\; .
$$

On the other hand, for $s=t-h$, since $t$ is the minimal time for which the
maximum $M_\alpha$ is achieved, we have $u_\alpha (y,s)-v(y,s) <M_\alpha$ and
$\beta(\bar C(s-t)+(|y-x|^2+1)^{1/2}-1)\geq -\beta \bar C h$. Since $h$ is
fixed, choosing $\beta$ small enough, we have $$ \max_{\overline
    \Omega}(u_\alpha (y,t-h)-\beta(-\bar
Ch+(|y-x|^2+1)^{1/2}-1)-v(y,t-h))<M_\alpha\; .$$

This shows that $\max_{\partial_p Q_{r,h}}(u_\alpha (y,s)-\beta(\bar
C(s-t)+(|y-x|^2+1)^{1/2}-1)-v(y,s))< M_\alpha$, a contradiction since
$u_{\alpha}(x,t)-v(x,t)=M_\alpha$. Therefore the maximum of $u_{\alpha}-v$ is
achieved either on $\partial \Omega$ or for $t=0$ and the complete comparison
result is obtained by letting $\alpha$ tend to $0$.  \end{proof}

In the direction of getting local comparison, we use below that
under \Q{k} we have a partial local comparison result 
for any ball which does not intersect the $\Man{j}$ for $j<k$
\begin{proposition} Let $B$ be a ball in $\R^N$ such that
    $B\cap\Man{j}=\emptyset$ for any $j<k$. If \Q{k} holds, then one has a
    comparison between sub and supersolutions of \HJBSD in $B\times
    (t_1,t_2]$ for any $0\leq t_1 <t_2 \leq T$, namely $$\| ( u-
    v)_+\|_{L^\infty(\bar Q)}\leq
    \| (u- v)_+\|_{L^\infty(\partial_p Q)}\,,$$
    where $Q:=B\times (t_1,t_2)$.
\end{proposition}

\begin{proof} For any $\eta >0$, $u-\eta t$ is a strict subsolution of \HJBSD
    in $B\times (t_1,t_2]$. Looking at a maximum point of $(u-\eta t) -
    v$ in $\bar Q$, we see that \Q{k} implies that such a maximum point cannot
    be in $B\times (t_1,t_2]$. Therefore all the maximum points are on
    $\partial_p Q$ and therefore 
    $$\| ( u-\eta t- v)_+\|_{L^\infty(\bar Q)}\leq
    \| (u-\eta t - v)_+\|_{L^\infty(\partial_p Q)}\,.$$
    Letting $\eta$ tends to $0$ provides the result.  
\end{proof}

\subsection{Properties of sub and supersolutions}

A consequence of the partial local comparison result deriving from \Q{k}
is a sub-dynamic programming principle for subsolutions

\begin{lemma}\label{lem:sub.Qk.dpp}
Let $u$ be an usc subsolution of \HJBSD and assume that \Q{k} is true for
some $k\in\{0,...,N\}$. Then, for any bounded domain $\Omega\subset\R^N$ such that
    $\bar \Omega\cap\Man{j}=\emptyset$ for any $j< k$, the subsolution 
    $u$ satisfies a sub-dynamic programming principle in
    $\bar \Omega\times[0,T]$: namely, for any $\sigma\in[0,t]$, 
 \begin{equation}\label{ineq:sub.Qk.dpp}
        u(x,t)\leq \inf_{(\X,L)\in\cT(x,t)}
        \sup_{\theta\in\mathcal{S}(\Omega)}\int_0^{\theta\wedge \sigma}
        l\big(\X(s),t-s\big)\ds+u\big(\X(\theta\wedge
        \sigma),t-(\theta\wedge\sigma)\big)\,,
    \end{equation}
where  $\mathcal{S}(\Omega)$ is the set of all stopping times $\theta$
such that $\X(\theta)\in\partial \Omega$.
\end{lemma}
It is clear that if $\tau_\Omega:=\sup\big\{s >0: \X(s)\in \Omega\big\}$ is the
first exit time from $\Omega$ and $\tau_{\bar \Omega}:=\sup\big\{s >0: \X(s)\in
    \bar \Omega\big\}$ the first exit time from $\bar \Omega$, we have
$\tau_\Omega\leq\theta\leq\tau_{\bar \Omega}$.

\begin{proof}
    Since $u$ is usc, we can approximate it by a decreasing sequence $\{u_n\}$ of
    continuous functions. Then we consider initial-boundary value problem
    (associated to an exit time control problem) 
    $$\begin{cases}
      w_t+H(x,t,Dw)=0 & \text{in }Q:=\Omega\times(0,T)\,,\\
      w(x,0)=u_n(x,0) & \text{on }\bar \Omega\,,\\
      w(x,t) = u_n(x,t) & \text{on } \partial \Omega\times(0,T)\,.
    \end{cases}$$
Since $u$ is an usc subsolution of \HJBSD and since $u \leq u_n$ on
$\partial_p Q$, then $u$ is a subsolution of this problem. On the other hand,
using that \Q{k} is true, the arguments of \cite{Ba} (Section 5.1.2, see Thm 5.7)
show that
$$ w_n(x,t):= \inf_{(\X,L)\in\cT(x,t)}
        \sup_{\theta\in\mathcal{S}(\Omega)}\bigg[\int_0^{\theta\wedge \sigma}
        l\big(\X(s),t-s\big)\ds+u_n\big(\X(\theta\wedge
        \sigma),t-(\theta\wedge\sigma)\big)\bigg]\,,
$$
is the maximal subsolution (and solution) of this initial value problem.
Therefore $u \leq w_n$ in $\bar Q$.  In order to obtain the result, we choose
any (fixed) trajectory $\X$ and any
cost $L$ and write that, by the above inequality 
$$ 
u(x,t) \leq\int_0^{\theta_n\wedge \sigma}
        l\big(\X(s),t-s\big)\ds+u_n\big(\X(\theta_n \wedge
        \sigma),t-(\theta_n \wedge\sigma)\big)\,,
$$
where $0\leq \theta_n \leq T$ is the stopping time where the supremum is
achieved. But $\theta_n$ is bounded and $\X(\theta_n \wedge \sigma)\in \partial
\Omega$ which is compact. Therefore extracting some subsequence we may assume
that $\theta_n \to \bar\theta$ and $\X(\theta_n \wedge  \sigma) \to \X(\bar
\theta \wedge \sigma)$. But, by using that $(u_n)_n$ is a decreasing sequence,
it is easy to prove that
$$
\limsup_n\, u_n\big(\X(\theta_n \wedge \sigma),t-(\theta_n\wedge\sigma)\big)
\leq u\big(\X(\bar \theta \wedge \sigma),t-(\theta\wedge\sigma)\big)\; ,
$$ 
and therefore
$$ 
  u(x,t) \leq \int_0^{\bar \theta\wedge \sigma}
  l\big(\X(s),t-s\big)\ds+u\big(\X(\bar\theta \wedge
  \sigma),t-(\bar\theta \wedge\sigma)\big)\,,
$$
Passing to the supremum in the right-hand and using that this is true for any
choice of $\X,L$ yields the result. 
\end{proof}

\begin{remark} 
    It is worth pointing out that, if $x\in \Omega$ then there exists $\eta >0$
    such that $\tau_\Omega >\eta$ for any trajectory $\X$. This is a consequence
    of the boundedness of $\BL$. Therefore, if we take $\sigma <\eta$, we
    clearly have
    $$ u(x,t)\leq \inf_{(\X,L)\in\cT(x,t)}\left\{
      \int_0^{\sigma} l\big(\X(s),t-s\big)\ds+u\big(\X(
        \sigma),t-(\sigma)\big)\right\}\,,
    $$
    a more classical formulation of the Dynamic Programming Principle.
\end{remark}

The next step is the
\begin{lemma}\label{RegTan} Let $x$ be a point in $\Man{k}$ and $t,h >0$. 
    There exists $r'>0$ such that if $u$ is a subsolution of \HJBSD in $B(x,r') \times
    (t-h,t)$, then for any $a\in(0,r')$, there exists a sequence of usc functions
    $(u^\e)_\e$ in $\overline{B}(x,r'-a) \times (t-h/2,t)$ such that \\[2mm] 
    $(i)$ the $u^\e$ are subsolutions of \HJBSD in $B(x,r'-a) \times
    (t-h/2,t)$,\\[2mm]
    $(ii)$  $\limssup u^\e = u$.
    \footnote{\,We recall that $\displaystyle \limssup u^\e (x,t)=
        \limsup_{\HRL} u^\e (y,s)$.}\\[2mm]
    $(iii)$ The restriction of $u^\e$ to
    $\Man{k}\cap \Big[B(x,r'-a)  \times
    (t-h/2,t)\Big]$ is $C^1$.
\end{lemma}

\begin{proof}
The proof is strongly inspired from \cite{BBC2}, with the
additional use of Lemma~\ref{flat-lem}. In fact, by using the definition of a
regular stratification (Definition~\ref{def:RS}), we can prove the result for
$\tilde u(y):=u(\Psi^{-1}(y))$ in the case of an \AFS and then make the
$\Psi$-change back to get the $u^\e$'s in the real domain.

Therefore, from now on, we assume that we are in the case of an \AFS and we still
denote by $u$ the function $\tilde u$ which is defined above. We are also going
to assume that $k\geq 1$ and we will make comments below on the easier $(k=0)$--case.

    We first pick a $r_0>0$ small enough so that $r_0<r(x)$ as in the Definition
    of Regular Stratifications, Definition~\ref{def:RS}. Then we take
    $0<r'<r_0$ so that $\Psi^x(B(x,r'))\subset B(\Psi^x(x),r)$ where $r$
    is defined in \AFS-$(i)$. This way, we make sure that we can use
    Lemma~\ref{flat-lem}, which will be needed below.

The next step is a sup-convolution in the $\Man{k}$-direction. Without loss of
generality, we can assume that $x=0$, and writing the coordinates in $\R^N$ as
$(y_1,y_2)$ with $y_1\in \R^k$, $y_2 \in \R^{N-k}$ we may assume that
$\Man{k}:=\{(y_1,y_2):\ y_2=0\}$.

With these reductions, the sup-convolution in the $\Man{k}$ directions (and also
the time direction) can be written as 
$$
u^{\e_1, \alpha_1}_1(y_1,y_2,s):=\max_{z_1 \in \R^k, s' \in
    (t-h,t)}\Big\{u(z_1,y_2,s')-\exp(Kt) 
    \left(\frac{|z_1-y_1|^2}{\e_1^2}+\frac{|s-s'|^2}{\alpha_1^2}\right)\Big\}\;,
$$
for some large enough constant $K>0$. We point out here that in the case $k=0$,
this sup-convolution is done in the $t$-variable only, as for the 
the usual convolution below.

By classical arguments, the function
$u^{\e_1, \alpha_1}_1$ is Lipschitz continuous in $y_1$ and $s$ and the normal
controllability assumption implies both that $u^{\e_1, \alpha_1}_1$ is Lipschitz
continuous in $y_2$ and allows to prove that, for $\e_1$ small enough and
$\alpha_1\ll \e_1$, $u^{\e_1, \alpha_1}_1(y,s)-c(\e_1,\alpha_1)s$ is still a
subsolution of \HJBSD in $B(x,r'-a/2) \times (t-3h/4,t)$ for some constant
$c(\e_1,\alpha_1)$ converging to $0$ as $\e_1\to 0$ and $\alpha_1\to 0$ with
$\alpha_1 \ll \e_1$,  we refer to \cite{BBC1,BBC2} for more
    details.
We point out that we need different parameters for this
sub-convolution procedure in space and in time because of the different
regularity of the Hamiltonian $H^k$ in space and time: while we require some
Lipschitz continuity in $y_1$ (up to the $|p|$-term), we have only the
continuity in $s$.

In this last statement, Lemma~\ref{flat-lem} plays a key role since it can be
translated as: if $(y_1,y_2) \in \Man{l} \cap B(x,r')$ for some $l\geq k$, then 
in the sup-convolution, the points $(z_1,y_2)$ with $z_1\in\R^k$ which are in
$B(x,r')$ still belong to $\Man{l}$. In other words, if $(y_1,y_2)\in\Man{l}$, 
for $\eps_1$ small enough the sup-convolution only takes into account values of 
$u$ taken on $\Man{l}$.

Thus, checking the subsolution condition $\tilde H^l\leq0$ for
$u^{\eps_1,\alpha_1}_1-c(\e_1,\alpha_1)t$ at $(y_1,y_2)\in\Man{l}$, is done by
considering the similar subsolution condition $\tilde H^l\leq0$ for $u$ at
points $(z_1,y_2)\in\Man{l}$.  We drop the details since the proof follows
classical arguments.

Next we regularize $u^{\eps_1,\alpha_1}_1-c(\e_1,\alpha_1)t$ by a standard
mollification argument, and still in the $(y_1,s)$-variables. If
$(\rho_{\e_2})_{\e_2}$ is a sequence of mollifiers in $\R^{k+1}$, $\rho_{\e_2}$
having a support in $B_k(0,\e_2)\times (-\e_2,0)$, where $B_k(0,\e_2)$ is the
ball of center $0$ and radius $\e_2$ in $\R^k$, we set 
$$
u^{\e_2}_2 (y_1,y_2,s):= \int_{\R^{k+1}}\, [u^{\e_1, \alpha_1}_1(z_1,y_2,s')-c(\e_1,\alpha_1)s']
\rho_{\e_2}(y_1-z_1,s-s')dz_1ds'\; . $$
By standard arguments, $u^{\e_2}_2$ is $C^1$ in $y_1$ and $s$, for all $y_2$ and
by the same argument as above, this convolution is done $\Man{l}$ by $\Man{l}$
(or $H^l$ by $H^l$: there is no interference between $H^l$ and $H^{l'}$ for
$\eps_2$ small enough), and, for $\e_2$ small enough, $u^{\e_2}_2
(y_1,y_2,s)-d(\e_2)s$ is still a subsolution of \HJBSD for some $d(\e_2)$
converging to $0$ as $\e_2 \to 0$; hence the proof is classical.

The conclusion follows from the fact that $u^{\e_2}_2 \to u^{\e_1, \alpha_1}_1$
uniformly as $\e_2 \to 0$ and $u^{\e_1, \alpha_1}_1 \downarrow u$ as $\e_1 \to
0$. Therefore we can take, for $u^\e$, $u^{\e_2}_2-d(\e_2)s$ with $\e_1$ small
and $\e_2$ small compared to $\e_1$.
\end{proof}

Concerning the supersolutions now, a key argument that was used in
\cite{BBC1,BBC2} is that they satisfy
an alternative: either the trajectories are leaving the discontinuity set, or
there is a strategy which allows to remain on this set and we deduce an
inequation for the tangential Hamiltonian there.

The situation is more complex here since the discontinuity set is composed of
submanifolds with different dimensions. But still, a similar  alternative can be
derived. In order to formulate it, let us introduce some notations.

Consider a point $x_0\in\Man{k}$ for some
$k\in\{0,...,(N-1)\}$, $t_0\in[0,T]$ and a sequence $(x_n,t_n)\to (x_0,t_0)$.  For any
$j\in\{0,...,(N-1)\}$, we denote by $\tau_n(j)$ the reaching time 
$$
\tau_n(j):=\inf\{s\geq0:X_{x_n,t_n}(s)\in\Man{j}\}\,
$$
where $(X_{x_n,t_n},L_{x_n,t_n})$ is a given solution of the differential
inclusion such that $(X_{x_n,t_n}(t_n),L_{x_n,t_n}(t_n))=(x_n,0)$.  Notice that
by \AFS-$(i)$, $\tau_n(j)>0$ for any $j<k$. 

\begin{lemma}\label{prop:super.alternative}
Let $v$ be a bounded lsc viscosity supersolution of \HJBSD and $\phi \in
C^{1}(\R^N\times (0,T])$ be a test-function such that the restriction of
$v-\phi$ to $\Man{k}\times(0,T]$ has a local minimum point at
$(x,t)\in\Man{k}\times(0,T]$. Then the following alternative
holds

\noindent{\bf A)} either there exists $\bar\tau>0$, a sequence $(x_n,t_n)\to
(x,t)$ and a sequence of trajectories $(X_{x_n,t_n},L_{x_n,t_n})$ satisfying
$\tau_n(j)\geq\bar\tau$
for any $j\leq k$ and
$$v(x_n,t_n)\geq \int_0^{\bar\tau} l\big(X_{x_n,t_n}(s),t_n-s\big)\ds +
v\big(X_{x_n,t_n}(\bar\tau),t_n-\bar\tau\big)\,;$$
{\bf B)} or $v_t(x,t)+H^k\big(x,t,Dv(x)\big)\geq0$ in the viscosity sense\,.
\end{lemma}

\begin{proof} 
    Since the result is local, we can prove it only in the case of an
    \AFS, the general result being obtained by changing variables. Therefore, we
    assume in particular in the sequel that $\Man{k}$ is a subspace and even
    that
    $$\Man{k}=\{x\in \R^N;\ x_{k+1}=x_{k+2}=\cdots =x_N=0\}.$$

    If $(x,t)\in\Man{k}$ is a local minimum point of $v-\phi$ on
    $\Man{k}\times(0,T)$, we can assume that it is a strict local minimum point
    by standard arguments. As we already noticed, \hyp{\bf M}-$(iv)$
    implies that there exists $\tau_0>0$ such that
    $\tau_n(j)\geq\tau_0$ for all $j<k$. In order to ``push'' the minimum point
    away from $x\in\Man{k}$, we construct the following test-function
    $$
    \phi_\eps(z,s):=\phi(z,s)+q\cdot(z-x)-\frac{{\rm
    dist}(z;\Man{k})^2}{\eps^2}\,,
    $$
    where $\eps >0$ and $q \in (\Man{k})^\bot$, where $(\Man{k})^\bot$ is the
    vector space which is orthogonal to $\Man{k}$. We point out that
    $(\Man{k})^\bot$ can be identified with $\R^{N-k}$.

    In order to choose $q$, we introduce the function $\chi:\R^N\to\R$ defined by
    $$\chi(q):=\phi_t(x,t)+ H\big(x,t,D\phi(x,t)+ q\big)\,,$$
    which is convex and coercive in $\R^N$. In fact, we are interested in the
    restriction of $\chi$ to $(\Man{k})^\bot$ and we denote by $\varphi
   :=\chi_{/(\Man{k})^\bot}$. If the minimum of $\varphi$ is achieved at
    $\bar{q}\in (\Man{k})^\bot$, then the classical property for the
    subdifferential of a convex function at a minimum point ($0\in \partial
    \varphi(\bar{q})$) can be reinterpreted here as 
    $$ \Man{k} \cap \partial \chi(\bar{q})\neq \emptyset\; ,$$
    since $\partial(\chi_{/(\Man{k})^\bot})=(\partial\chi)_{/(\Man{k})^\bot}\;
    .$ This fact can easily be proved using the identification between
    $(\Man{k})^\bot$ and $\R^{N-k}$, the fact that, in $\R^{N-k}$, we have $0$
    in the subdifferential of $\varphi$ can be interpreted as the existence of
    an element in $\partial \chi(\bar{q})$ which is in $\Man{k}$.

    Finally, taking into account the definition of $H$, the fact that $\BL(x,t)$
    is convex and classical result on convex function, namely Danskin's Theorem which has to be
    translated again from $\R^{N-k}$ to $(\Man{k})^\bot$, then there exists
    $(b,l)\in\BL(x,t)$ such that 
    $$ \chi(\bar{q}) = \phi_t(x,t)-b\cdot (D\phi(x,t)+\bar q) - l \; ,$$
    and $b\in \Man{k}\cap \partial \chi(\bar{q})$.
 
    But we are in the \AFS case where $T_x\Man{k} = \Man{k}$ and the above
    property yields 
    $$\phi_t(x,t)+H^k(x,t,D\phi(x,t))=
      \phi_t(x,t)+ \sup_{\substack{(b,l)\in\BL(x,t)\\ b\in T_x\Man{k}}}\{-b\cdot
      D\phi(x,t)-l\}\geq \varphi(\bar{q})\,.
    $$
    If $\varphi(\bar{q}) \geq 0$, then {\bf B)} holds and we are done. Hence we
    may assume that $\varphi(\bar{q}) <0$.

    From now on we consider the function $\phi_\eps$ with the choice
    $q=\bar{q}$. Notice that, in this case $\phi_\eps=\phi$ on
    $\Man{k}\times(0,T]$: the distance term clearly vanishes and $\bar q$ is
    orthogonal to $z-x$ if $z\in\Man{k}$.

    Since $(x,t)$ is a strict local minimum point of $v-\phi$ on $\Man{k}\times
    (0,T)$, there exists a sequence $(x_\eps,t_\eps)$ of local minimum points of
    $v-\phi_\eps$ in $\R^N\times (0,T)$ which converges to $(x,t)$.  There are
    two possibilities.

    \bigskip

    \noindent \textbf{First case:} assume that for $\eps >0$ small enough,
    $(x_\eps,t_\eps)\in\Man{k}\times(0,T)$.
    
    On the one hand, $(v-\phi)$ and $(v-\phi_\eps)$ coincide
    on $\Man{k}\times(0,T)$ and $(v-\phi)$ has a strict local minimum at
    $(x,t)$, say in $V(x,t):=B(x,\eta)\times(t-h,t+h)$.
    On the other hand, $(v-\phi_\eps)$ has a local minimum at
    $(x_\eps,t_\eps)$ which converges to $(x,t)$. Hence, for $\eps$ small
    enough, $(x_\eps,t_\eps)\in V(x,t)$ and we deduce that necessarily for such
    $\eps$, $(x_\eps,t_\eps)=(x,t)$ by the strict local minimum property.

    Then, writing the supersolution viscosity inequality reads
    $$ 
    0\leq \phi_t(x,t)+ H\big(x,t,D\phi(x,t)+ \bar{q}\big) = \varphi(\bar{q})<0\; ,
    $$
    which is a contradiction.

    \bigskip

    \noindent\textbf{Second case:} there exists a subsequence of
        $(x_{\eps},t_\eps)$ denoted by $(x_n,t_n)$ such that $x_n\notin\Man{k}$.
    
    \noindent\textsc{Step 1 ---}  Notice first that necessarily we have
    $\tau_n(k)\to0$.  Thus, between times $t=0$ and $t=\tau_n(k)$, $X_{x_n,t_n}(s)$
    remains inside a ball $B\subset\R^N$ such that 
    $B\cap\Man{j}=\emptyset$ for any $j\leq k$. By  
    Lemma~\ref{lem:super.dpp} we can use the super dynamic programmation principle for
    $v(x_n,t_n)$ between times $0$ and
    $\tau_n\wedge\tau_B$, where we write $\tau_n$ for $\tau_n(k)$. 
    Taking $n$ large enough so that $\tau_n<\tau_B$, we get
    \begin{equation}\label{eq:rep.v}
        \frac{v(x_n,t_n)-v\big(X_{x_n,t_n}(\tau_n),t-\tau_n\big)}{\tau_n}\geq\frac{1}{\tau_n}
        \int_0^{\tau_n} l(X_{x_n,t_n}(s),t_n-s)\ds\,.
    \end{equation}
    Since $(X_{x_n,t_n},L)$ satisfies the differential inclusion, we have
    $$X_{x_n,t_n}(\tau_n)={x_n}+\int_0^{\tau_n}
    b(X_{x_n,t_n}(s),t_n-s)\ds$$ 
    for some function $b$ such that for any $s\in(0,\tau_n)$,
    $b(X_{x_n,t_n}(s),t_n-s)\in\B(X_{x_n,t_n}(s),t_n-s)$.
    Hence, taking the test-function $\phi_\eps$ we have, on one hand
    $$
    v(x_n,t_n)-v\big(X_{x_n,t_n}(\tau_n),t-\tau_n\big) \leq \phi_\eps(x_n,t_n)-
    \phi_\eps\big(X_{x_n,t_n}(\tau_n),t_n-\tau_n\big)\; ,
    $$
    and, on the other hand
    $$
    {\rm dist}(x_n;\Man{k})^2 - {\rm dist}\big(X_{x_n,t_n}(\tau_n);\Man{k})^2 = {\rm
        dist}(x_n;\Man{k})^2 \geq 0\; .
    $$
    Therefore 
    \begin{equation}\label{eq:rep.phi}
        \phi_\eps\big(X_{x_n,t_n}(\tau_n),t_n-\tau_n\big)-\phi_\eps(x_n,t_n) \geq
        \partial_t\phi (x_n,t_n)\tau_n + (D\phi(x_n,t_n) +\bar{q})\cdot
        \int_0^{\tau_n} b(X_{x_n,t_n}(s),t_n-s)\ds+o(\tau_n)\,.
    \end{equation}
    Combining \eqref{eq:rep.v} with the above properties, we get
    \begin{align}
        \partial\phi(x_n,t_n)\geq \label{eq:phi}
        \big(D\phi(x_n,t_n)+\bar q\,\big)+ & \frac{1}{\tau_n}\int_0^{\tau_n} 
        b(X_{x_n,t_n}(s),t_n-s)\ds \\ \nonumber &+ \frac{1}{\tau_n}\int_0^{\tau_n}
        l(X_{x_n,t_n}(s),t_n-s)\ds+o(1)\,.
    \end{align}

    \noindent\textsc{Step 2 ---}
    By the properties \hyp{\BL}, we claim that there exists a
    couple $(b,l)\in\BL(x,t)$ such that, at least along a subsequence
    $$
    \frac{1}{\tau_n}\int_0^{\tau_n} b(X_{x_n,t_n}(s),t_n-s)\ds\to b\,,\ 
    \frac{1}{\tau_n}\int_0^{\tau_n} l(X_{x_n,t_n}(s),t_n-s\big)\ds\to l\,.
    $$
    Indeed, notice first that as $\tau_n\to0$, $X_{x_n,t_n}(\cdot)\to x$
    and $(t_n-\cdot)\to t$, both convergences being uniform on $[0,\tau_n]$.
    
    Then, there exists a sequence $\eps_n\to0$ and
    $$(b_n,l_n) \in(\BL)_{\eps_n}(x,t):=\overline{\mathrm{co}}\,\Big(
    \bigcup_{\substack{|z-x|\leq\eps_n \\ |s-t|\leq \eps_n}}\BL(z,s)\Big)
    $$
    such that
    $$
    \frac{1}{\tau_n}\int_0^{\tau_n}b(X_{x_n,t_n}(s),t_n-s)\ds=b_n\quad , \quad
    \frac{1}{\tau_n}\int_0^{\tau_n} l(X_{x_n,t_n}(s),t_n-s\big)=l_n
    \,.
    $$
    By the bounds for $b_n,l_n$, we deduce that at least along a
    subsequence still denoted by $b_n,l_n$, we have $(b_n,l_n)\to (b,l)$ for
    some $(b,l)\in\R^N\times \R$. Now, since the images of the $\BL(z,s)$ are
    convex and since $\BL$ is upper semi-continuous,
    $\distH((\BL)_{\eps_n}(x,t),\BL(x,t))\to0$ as $\eps_n \to 0$ and we deduce
    that $(b,l)\in \BL(x,t)$.
        
    \noindent\textsc{Step 3 ---} Passing to the limit in \eqref{eq:phi} as
    $\tau_n\to0$ yields
    $$\partial_t\phi(x,t)\geq \big(D\phi(x,t)+\bar q\,\big)\cdot b + l\,.$$
    But this is in contradiction with the assumption that $\varphi(\bar{q})<0$.
    Hence, either \textbf{A)} holds or this second case cannot happen
    and then \textbf{B)} holds. This ends the proof.
\end{proof}

\subsection{Proof by induction on the dimension of $\Man{k}$}

As we already noticed above, \Q{N} necessarily holds true since in this case the
ball does not intersect any discontinuity. Moreover, we proved in \cite{BBC2}
that \Q{N-1} is also true. Of course, \Q{0} means that we have a comparison
result without any restriction on the submanifolds $\Man{k}$ which intersects
$B(x,r)$.  Thus, the proof of Theorem~\ref{thm:main} is reduced to the following
backwards induction property

\begin{proposition}\label{prop:induction.tilde}
    Assume that \Q{k} is true for some $k\in\{1,...,N-1\}$. Then \Q{k-1} is
    also true.
\end{proposition}
\begin{proof}  We consider a ball $B\subset \Man{k-1}\cup \Man{k} \cup\cdots\cup \Man{N}$, $0\leq t_1 <t_2 \leq T$, an
    usc function $u$ which is a strict subsolution  of \HJBSD in $B\times
    (t_1,t_2]$ and a lsc supersolution $v$ of \HJBSD in $B\times (t_1,t_2]$.

In order to check \Q{k-1} we have to show that $u-v$ cannot have a maximum point $(\bar x, \bar t)$ in $B\times (t_1,t_2]$.
But by \Q{k}, $\bar x$ cannot belong to
    any $\Man{j}$ for $j\geq k$. Therefore, we are left with the case where
    $\bar x\in\Man{k-1}$.
Using \RS and \AFS-$(i)$,  we consider a smaller ball $B'$ such that $\bar
B'\subset B$ 
still containing $\bar x$ and such that
$B'\cap\Man{j}=\emptyset$ for any $j<k-1$. 

Using that the Hamiltonians $H$ and $H^{j}$ are Lipschitz continuous in $p$,
we can replace $u$ by $\bar u(x,t):=u(x,t)-\delta ((t-\bar t)^2+|x-\bar
x|^2)$ for $\delta >0$ small enough: this new function is still a strict
subsolution and $(\bar x, \bar t)$ is a strict local maximum point of $\bar u
-v$.

Next we use Lemma~\ref{RegTan} for the subsolution $\bar u$ and for $r,h>0$ small
enough: since there exists a sequence $(u^\e)_\e$ of subsolutions such that
$\limssup u^\e = \bar u$, there exists an usc subsolution $u_\flat$ defined in
$B(\bar x,r)\times (\bar t-h,\bar t)\subset B'\times(t_1,t_2)$ and a maximum
point $(x_\flat,t_\flat)$ of $u_\flat-v$ which is also as close as we want to
$(\bar x,\bar t)$.

We can therefore assume that $x_\flat\in B'$ and since \Q{k} holds true, 
necessarily $x_\flat\in\Man{k-1}$ for the same reason as for $\bar x$ above.

Consider now Lemma~\ref{prop:super.alternative} for $v$ at $x_\flat$. If we are in
case \textbf{A)} of the alternative, we get a sequence $x_n\to x_\flat$ which
remains in $\Omega:=B'\setminus \Man{k-1}$, and $\bar\Omega$  does not intersect any
    $\Man{j}$ for $j\leq k-1$. Moreover, the reaching times of trajectories issued
from the $x_n$ are controled from below.

Next, we use in conjunction Lemma~\ref{lem:sub.Qk.dpp} in $\Omega$: the 
sub-optimality principle satisfied by $u_\flat$ in $\Omega$ implies that
for some $\sigma\in(0,h)$ small enough (but uniform with respect to $n$)
$$ u_\flat(x_n,t_n)- v(x_n,t_n) \leq
    u_\flat (X_{x_n,t_n}(\sigma),t_n-\sigma)-
    v(X_{x_n,t_n}(\sigma),t_n-\sigma)-\eta\sigma\,,$$
where $\eta$ comes from the strict subsolution property for $u_\flat$.
    Passing to the limit as $x_n\to x_\flat$ we obtain
    $$ u_\flat (x_\flat,t_\flat) - v(x_\flat,t_\flat) \leq
    u_\flat (X_{x_\flat,t_\flat}(\sigma),\bar t-\sigma)-
    v(X_{x_\flat,t_\flat}(\sigma),\bar t-\sigma)-\eta\sigma \,.$$
and this contradicts the fact that $(x_\flat,t_\flat)$ is a local maximum
point of $u_\flat -v$.

In case \textbf{B)}, since by  Lemma~\ref{RegTan} $u_\flat$ is $C^1$ on $\Man{k}$, we have 
$$ u_\flat(x_\flat,t_\flat)_ t + H^k(x,t,Du_\flat(x_\flat,t_\flat)) \geq 0 \; .$$
But this is also a contradiction since $u_\flat$ is a strict subsolution and therefore
$$u_\flat(x_\flat,t_\flat)_ t + H^k(x^\e,t^\e,Du_\flat(x_\flat,t_\flat)) \leq -\eta <0
\,.$$
Hence, such a maximum point $(x_\flat,t_\flat)$ cannot exist, which implies that
if $\bar x$ exists, it has to be located on $\Man{j}$ for some $j<k-1$, thus and \Q{k-1}
holds true.
\end{proof}

\section{A Stability Result}\label{NESR}

In  this section we prove a stability result when we have a sequence of problems
on stratified domains $\HJBSD_\eps$.  An important issue here is that, not only
do the corresponding Hamiltonians depend on $\eps$, but also the stratification
of space does. More precisely, for each $\eps >0$ we are given a regular
stratification $\M_\eps$ and a notion of convergence is required.

This is the purpose of the following definition.
\begin{definition}\label{def:conv.strat} 
    We say that a sequence $(\M_\eps)_\eps$ of regular stratification of $\R^N$.
    converges to a regular stratification $\M$ if, for each $x\in\R^N$, there exists
    $r>0$, an \AFS $\M^\star=\M^\star(x,r)$ in $\R^N$ and, for any $\eps >0$, changes of
    coordinates $\Psi^x_\eps, \Psi^x $ as in Definition~\ref{def:RS} such that
    $\Psi^x_\eps(x)= \Psi^x(x)$ and \\[2mm]
    $(i)$ $\Psi^x_\eps(\Man{k}_\eps\cap B(x,r))= \M^\star \cap
    \Psi^x_\eps(B(x,r))$, $\Psi^x(\Man{k}\cap B(x,r))=
    \M^\star\cap\Psi^x(B(x,r))$.\\ $(ii)$ the changes of coordinates
    $\Psi^x_\eps$ converge in $C^{1}(B(x,r))$ to $\Psi^x$ and their inverses
    $(\Psi^x_\eps)^{-1}$ defined on $\Psi^x(B(x,r))$ also converge in $C^1$ to
    $(\Psi^x)^{-1}$.\\[2mm]
    We denote this convergence by $\M_\eps \toRS \M$.    
\end{definition}

Thus, the manifolds $\Man{k}_\eps$ ($k=0..N$) can vary with $\eps$ but after suitable
changes of variable $\Psi^x_\eps$, they are flat and constant. The important issue
is that we do not want to create/destroy/intersect manifolds when they move.

Then we also consider, for each $\eps >0$, the associated
Hamilton-Jacobi-Bellman problem in the stratified domain $\M_\eps$, that we
denote by $\HJBSD_\eps$. The meaning of sub and supersolutions is the one that
is introduced in Definition~\ref{def:HJBSD}, with the family of Hamiltonians
$H_\eps$ and $(H^k_\eps)$ that are constructed from $\M_\eps$ and some family
$\BL_\eps$. 

In order to formulate the following stability result, we have to define limiting
Hamiltonians for the $H^k_\eps(x,t,p)$ but the difficulty is that they are
defined for $x\in \Man{k}_\eps$ which depends on $\eps$. In order to turn around
this difficulty, we use the change of variables of
Definition~\ref{def:conv.strat} which leads to consider the Hamiltonians
$\tilde{H}^k_\eps$, defined for $x\in \M^\star\cap\Psi^x(B(x,r))$, a
domain which does not depend on $\eps$. We make a slight abuse of notations by
saying that $H^k=\liminf_*H^k_\eps$ if the associated rectified
Hamiltonians satisfy $\tilde{H}^k=\liminf_* \tilde{H}^k_\eps$.

\begin{theorem}\label{thm:main.stability}
Assume that $(\M_\eps)_\eps$ is a sequence of \RS in $\R^N$ such that $\M_\eps
\toRS \M$, then the following holds 
  \begin{itemize} 
       \item[$(i)$] if, for all $\eps >0$, $v_\eps$ is a lsc supersolution of
            $\HJBSD_\eps$, then $ \underline{v}=\liminf_* v_\eps $ is a  lsc
            supersolution of  \HJBSD, the HJB problem associated with
            $H=\limssup \,H_\eps$.
      \item[$(ii)$] If, for $\eps >0$, $u_\eps$ is an usc subsolution of
          $\HJBSD_\eps$ and if the Hamiltonians $(H^k_\eps)_{k=0..N}$ satisfy
           \NC and \TC with uniform constants, 
          then $\bar{u} = \limsup^*u_\eps$ is a subsolution of
           \HJBSD with $H^k=\limiinf H^k_\eps$ for any $k=0..N$.
   \end{itemize}
\end{theorem}

\begin{proof} Result $(i)$ is standard since only the $H_\eps / H$-inequalities
    are involved and therefore $(i)$ is nothing but the standard stability result
    for discontinuous viscosity solutions with discontinuous Hamiltonians, see
    \cite{Idef}.

For $(ii)$, because of the definition of the convergence of the \RS, we can
assume without loss of generality that the \RS is fixed and is in fact an \AFS.
Then if $(x_0,t_0) \in \Man{k} \times (0,T)$ is a {\em strict} local maximum
point of $\bar u -\phi$ on $\Man{k}$, where $\phi$ is a $C^1$ function in
$\R^N$, we consider the functions $$u_\eps (x,t) -\phi (x,t) -
L\cdot\dist(x,\Man{k})$$ where $\dist(\cdot,\Man{k})$ denotes the distance to
$\Man{k}$.

For $\eps$ small enough, this function has a maximum point $(x_\eps,t_\eps)$ near
$(x_0,t_0)$. If $x_\eps \in \Man{l}$ for $l>k$, we have (because $u_\eps$ is an
usc subsolution of $\HJBSD_\eps$) $$ \phi_t (x_\eps,t_\eps) + H^l_\eps
\Big(x_\eps,t_\eps, D\phi(x_\eps,t_\eps) +  L\cdot
D\big[\dist(x_\eps,\Man{k})\big]\Big)\leq 0\;
.$$

Next we remark that, on the one hand, $D\big[\dist(x_\eps,\Man{k})\big] \in
V_k^\bot$ (recall that we are in the \AFS case) and on the other hand
$\big|D\big[\dist(x_\eps,\Man{k})\big]\big|=1$; therefore we can use \NC and
choose $L$ large enough in order that this inequality cannot hold. Notice that
this choice does not depend neither on $\eps$ nor on $l$, but we use that the
distance to $\Man{k}$ is smooth if we are not on $\Man{k}$.

Therefore $x_\eps \in \Man{k}$ for $l>k$, and $(x_\eps,t_\eps)$ is a local
maximum point of $u_\eps (x,t) -\phi (x,t)$ on
$\Man{k}$ (we can drop the distance since we look at the function only on
$\Man{k}$). Hence $$ \phi_t (x_\eps,t_\eps) + H^k_\eps \Big(x_\eps,t_\eps,
D\phi(x_\eps,t_\eps) \Big)\leq 0\; .$$

But using that $\bar{u} = \limsup^*u_\eps$ and that $(x_0,t_0)$ is a strict
local maximum point of $\bar u -\phi$ on $\Man{k}$, classical arguments imply
that $(x_\eps,t_\eps) \to (x_0,t_0)$ and the conclusion of the proof follows as
in the standard case.  \end{proof}

We conclude this section with some sufficient conditions on $\BL$ for the
stability of solutions.

\begin{lemma}  \label{lem:stab.bl.ham}
  For any $\eps >0$, let $\BL_\eps$  satisfy \hyp{\BL}, \TCBL and \NCBL
  with constants independent of $\eps$ and assume that $\M_\star$ is a fixed
  \AFS adapted to every $\BL_\eps$. Assume that $\BL_\eps\to\BL$ in the sense of
  the Haussdorf distance. Then for every $k\in\{0,...,N\}$, 
  $H^k_\eps\to H^k$ locally uniformly in $\Man{k}_\star\times(0,T)\times\R^N$.
\end{lemma}
\begin{proof}
    Since we are in a flat (and static) situation, let us first notice that the
  Hamiltonians $H^k_\eps$ are all defined on the same set. Then the convergence
  of $\BL_\eps$ implies that $(\BL_\eps)|_k$ (the restriction to
  $\Man{k}_\star\times[0,T]$) converges locally uniformly to $\BL|_k$. 
  It follows directly that
  $$H^k(x,u,p):=\sup_{\substack{(b,l)\in\BL_\eps(x,t) \\ b\in
          T_x\Man{k}_\star}}
      \big\{-b\cdot p-l \big\}\longrightarrow
      \sup_{\substack{(b,l)\in\BL(x,t) \\ b\in T_x\Man{k}_\star}}
      \big\{-b\cdot p-l \big\}
  = H^k(x,u,p)\,.$$
\end{proof}

\begin{corollary}\label{cor:stability}
  For any $\eps >0$, let $\BL_\eps$  satisfy \hyp{\BL} with constants independent
  of $\eps$, and consider an associated regular stratification
  $(\M_\eps,\Psi_\eps)$. We assume that $\BL_\eps\to\BL$ in the sense of
  Haussdorf distance and that $\M_\eps\toRS\M$. Let $U_\eps$ be 
  the unique solution of $\HJBSD_\eps$. Then
  $$U_\eps\to U\quad\text{locally uniformly in }\R^N\times[0,\infty)\,,$$
  where $U$ is the unique solution of the limit problem \HJBSD. 
\end{corollary}

\begin{proof}
    The proof is immediate: by the convergence of $\BL_\eps$ and $\M_\eps$,
    after a suitable change of variables we are reduced to considering the case
    of a constant local \AFS, $\M_\star$. Then we apply Lemma~\ref{lem:stab.bl.ham}
    which implies that the $(\tilde{H}^k_\eps)_k$ converge to the
    $(\tilde{H}^k)_k$. We invoke Theorem~\ref{thm:main.stability} which
    says that the half-relaxed limits of the $U_\eps$ are sub and
    supersolutions of the limit problem, \HJBSD. And finally, the comparison
    result implies that all the sequence converges to $U$.
\end{proof}

\section{Examples and extensions}

\subsection{Examples}

\noindent\textsc{Example 1:  a straight line in $\R^3$  ---} 
    \textit{This example is a typical example which is out of the scope of
        \cite{BBC1,BBC2} since the discontinuity set is not a
        $(N-1)$-dimensional manifold, but a lower dimensional one. We take the
        opportunity of this simple example to describe the way our assumption
        have to be read.}

We consider the line $\Gamma=\{x_1=x_2=0,
    x_3\in\R\}\subset\R^3$ and two bounded and continuous functions $(b,l)$ defined on
$(\R^3\setminus\Gamma)\times[0,\infty)\times A$ where $A$ is a control set.
We set as above $BL(x,t):=\{(b,l)(x,t,a):a\in A\}$ on
$(\R^3\setminus\Gamma)\times[0,\infty)\times A$ and 
$$\BL(x,t):=\begin{cases}
    BL(x,t) & \text{if }x\in\R^3\setminus\Gamma\,,\\
    \overline{\mathrm{co}}\Big(\limsup\limits_{y\to x\atop
        y\notin\Gamma}BL(x,t)\Big) & \text{if
    }x\in\Gamma\,.
\end{cases}$$
The natural stratification is simply $\Man{3}=\R^3\setminus\Gamma$,
$\Man{1}=\Gamma$ and $\Man{2}=\Man{0}=\emptyset$. An interesting point here is
the assumptions on $b,l$ which ensures \TCBL and \NCBL.

For \TCBL, the functions $b$ and $l$ have to be continuous in
$\R^3\setminus\Gamma \times [0,T] \times A$, $b$ being locally Lipschitz
continuous in $x$ with (locally) a uniform constant in $t$ and $a$. Of course,
they have to be bounded to have \hyp{\BL}. Moreover, in a neighborhood of each
point $(0,0,x_3)$, the functions $(x_3,t,a) \mapsto b((x_1, x_2,x_3), t,a)$ and
$(x_3,t,a) \mapsto l((x_1, x_2,x_3), t,a)$ are equicontinuous for $(x_1,x_2)$ in
a neighborhood of $(0,0)$ and, in the same way, the functions $x_3 \mapsto
b((x_1, x_2,x_3), t,a)$ are equi-Lipschitz continuous. In that way, if for any
sequence $(x^\e_1,x^\e_2)$ converging to $(0,0)$ such that
$$b((x^\e_1,x^\e_2,x_3), t,a)\to \bar b(x_3, t,a)\quad \hbox{ and }\quad
l((x^\e_1,x^\e_2,x_3), t,a)\to \bar l(x_3, t,a)\; ,$$ then $\bar b, \bar l$
satisfy classical assumptions, namely they are continuous and $\bar b$ is
locally Lipschitz continuous in $x_3$, uniform in $t$ and $a$. With this remark,
it is rather easy to show that $H^1$ defined on $\Gamma$ satisfies the right
continuity assumptions in $x_3$ and $t$.

In this example, it is clear that $x_3$ (and in a slightly different way $t$)
plays the role of the tangential derivatives while $(x_1,x_2)$ are the normal
ones.

For \NCBL, we write $b=(b_1, b_2,b_3)$ and the condition is that in a
neighborhood of each point $(0,0,x_3)$, there exists $\delta = \delta (x_3)$
such that $$ B(0,\delta) \subset \{(b_1,b_2)(x,t,a):a\in A\}\; ,$$
where $ B(0,\delta)$ is here a ball in $\R^2$.

Notice that, as we did it above in the checking of \TCBL, the dynamic and
cost on $\Gamma$ are obtained as the limits of the dynamic and cost on
$\R^3\setminus \Gamma$. But, of course, specific dynamic and cost can also
exist on $\Gamma$.

Under these conditions, we have a unique solution for \HJBSD.

\

\noindent \textsc{Example 2: the cross problem in $\R^2$ ---} 
    \textit{This example is another typical example which could not be treated in
        \cite{BBC1,BBC2}, with a more complex geometry: the discontinuity set
        contains an intersection of straight lines, that is, a point.}

In $\R^2$ we consider four domains as follows
$$\R^2=\big(\Omega_1\cup\Omega_2\cup\Omega_3\cup\Omega_4\big)\cup\Gamma\,,$$
where $\Gamma=\{x_1=0\}\cup\{x_2=0\}$ and each $\Omega_i$ is an open quadrant.
Then consider a control set $A$ and we 
assume that we have four vector fields $(b_i)_{i=1..4}$ and running costs
$(l_i)_{i=1..4}$, all bounded such that
$(b_i,l_i):\overline{\Omega_i}\times[0,\infty)\times A\to(\R^2\times\R)$
is continuous with respect to the first two variables, $(x,t)$.  

We then define the associated stratification $\Man{2}:=\cup_{i=1}^4\Omega_i$,
$\Man{1}:=\{x_1>0,x_2=0\}\cup\{x_1<0,x_2=0\}\cup\{x_1=0,x_2>0\}\cup\{x_1=0,x_2<0\}$
and finally $\Man{0}=\{(0,0)\}$. For $x\in\overline{\Omega_i}$, we set
$BL_i(x,t)= \big\{(b_i,l_i)(x,t,a): a\in A\big\}$ and finally
$$\BL(x,t):= \begin{cases}
    BL_i(x,t) & \text{if }x\in\Man{2}\,,\\
    \overline{\mathrm{co}}(BL_i(x,t)\cup BL_j(x,t)) & \text{if }x\in \Man{1}\\
    \overline{\mathrm{co}}(\cup_{i=1}^4 BL_i(0,t) & \text{if }x\in\Man{0}\,,
\end{cases}
$$
where of course the indices $i$ and $j$ are chosen accordingly to which portion
of $\Man{2}$ or $\Man{1}$ the point $x$ belongs to. With this setting we have a
\HJBSD which has a unique solution provided the assumptions on the $(b_i,l_i)$
are satisfied. These (local) conditions on $\Man{2},\Man{1}$ are analogous to
the ones described in Example 1. In a neighborhood of $(0,0)$, we need the
system to be fully controllable and the condition on $\Man{0}$ reduces to 
$$ u_t
\leq \inf\Big\{\sum_{i=1}^4 \mu_i l_i(0,t) : 
    \sum_{i=1}^4 \mu_i b_i(0,t) =0\Big\}.
$$

\

\noindent\textsc{Example 3: specific control problem on the discontinuity set ---}
 \textit{in this last example, we add specific control problems on the
        various submanifolds of positive codimension.}

We start from a continuous dynamic-cost map
$BL$ defined in $(\R^3\setminus\Gamma)\times[0,T]$, but we also put specific control
problems on $\Man{2}$, $\Man{1}$ and $\Man{0}$ according to the
stratification in $\R^3$ corresponding to Figure 1 (see Section \ref{sect:afs}). 

Hence, for $k=0,1,2$, we introduce a set-valued map $BL_k(\cdot,\cdot)$ which is
continuous on $\Man{k}\times[0,T]$. In order to have a global \HJBSD, we define
$\BL$ by setting 
$$\BL(x,t):=\begin{cases}
    BL(x,t) & \text{if }x\in\R^3\setminus(\Man{0}\cup\Man{1}\cup\Man{2})\,,\\
        \overline{\mathrm{co}}\Big(\limsup\limits_{y\to x\atop
        y\notin\Gamma} BL(x,t)\cup BL_k(x,t)\Big) & 
        \text{if }x\in\Man{k}\,,\ k=0,1,2\,.
\end{cases}
$$
The map ${\BL}$ satisfies \hyp{\BL}{} and provided each $BL_k$
and $BL$ satisfy \NCBL, we have an \HJBSD which has a unique solution.

\

\subsection{Applications \& Extensions}

\noindent\textsc{The Filippov approximation ---} a way to build a solution of
$u_t+H(x,u,Du)=0$ in $\R^N$ in presence of discontinuities consists in using the
Filippov approximation for the corresponding control problem: for each $\eps >0$ 
we consider 
$$\BL_\eps(x,t):=\overline{\mathrm{co}}\,
\bigg(\bigcup_{{|z-x|+|t-s|\leq\eps}}\Big(1-\frac{|z-x|}{\eps}-
\frac{|s-t|}{\eps}\Big)\,\BL(z,s)+
\Big(\frac{|z-x|}{\eps}+\frac{|s-t|}{\eps}\Big)\,\BL(x,t)\bigg)
\,.$$
The construction of $\BL_\eps$ comes from several considerations\\
$(i)$ for each $\eps >0$, $\BL_\eps$ is a continuous set-valued map with convex,
compact images;\\
$(ii)$ $\BL_\eps(x,t)$ also takes into account the specific dynamic-cost at
$(x,t)$;\\
$(iii)$ $\BL_\eps(x,t)$ takes into account dynamics-costs coming from a
neighborhood of $(x,t)$.

Notice first that by construction, $\BL_\eps$ is a continuous set-valued map
which satisfies \hyp{\BL}{} and \NCBL,\TCBL. Therefore there exists a unique solution
$U_\eps$ of $\HJBSD_\eps$, associated to $\BL_\eps$.

Since $\BL_\eps$ is continuous, if $\M$ is a stratification adapted to $\BL$, it
can be seen as a stratification adapted also to $\BL_\eps$, for any $\eps$,
even if there is no discontinuity for $\BL_\eps$.  Thus, $\BL_\eps\toRS\BL$ and
the stability result (Corollary \ref{cor:stability} yields that $U_\eps$
converges to the unique solution of \HJBSD).  This result extends \cite[Thm.
6.1]{BBC2} where the convergence of Filippov's approximation was proved for an
$(N-1)$-dimensional discontinuity set.

\

\noindent\textsc{Infinite horizon problems ---} we derived a complete study of
parabolic \HJBSD which correspond to finite horizon control problems. In the same way,
we can handle similarly the case of infinite horizon problems, leading to
stationary \HJBSD as in \cite{BBC1}.

This amounts to considering a set-valued map $x\mapsto \BL(x)$ and introduce the 
Hamiltonians $H^k(x,u,p)=\sup(\lambda u -p\cdot b-l)$, where the supremum is
taken over $\BL(x)$, with $b\in T_x\Man{k}$. The adaptations are quite
straightforward: under \TCBL,\NCBL (which have to be considered as independent
of $t$) we get comparison for the complemented problem; and the value function
of the associated control problem is the unique viscosity solution of this
complemented problem.

\

\noindent\textsc{Time-depending stratifications ---} throughout this paper, we
assumed that the discontinuities of the set-valued map $\BL(\cdot,\cdot)$ is
independent of the time-variable. This is a simplification which can be relaxed
at (almost) no cost in some situations: following the ideas of the stability
result, we assume that  for each $t>0$ we have a stratification $\M(t)$ adapted
to $\BL(\cdot,t)$ with the following property

\noindent \textit{for each $(x,t)\in\R^N \times [0,T]$, there exists $r>0$, an \AFS $\M^\star$ in
    $\R^N$ and a local change of coordinates $\Psi^{(x,t)}: B(x,r)\times (-r,r)
    \to \R^N \times \R$ as in Definition~\ref{def:RS} such that
    $$\Psi^{(x,t)}\Big(\Man{k}\cap \big(B(x,r)\times(-r,r)\big)\Big)=
    (\M^\star \times \R)\cap \Psi^{(x,t)}\Big(B(x,r)\times (-r,r)\Big)\,.$$}

    This means that, up to the local changes of variables $\Psi^{(x,t)}$, we are in a
flat and time-independent situation. All the constructions and results that we
derived thus apply with slight modifications. Notice that of course the
dependance of $\Psi^{(x,t)}$ with respect to the time variable should be regular
enough so that the rectified equation keeps the suitable properties \TC,\NC,\LP,
\textit{i.e.} $C^1$ in the $t$-variable, and $W^{2,\infty}$ or $C^1$ in the
$x$-variable (depending on the controllability assumptions).

\

\noindent\textsc{A word on the maximal solution ---} by focusing on the
complemented \HJBSD problem, the unique solution we select is the minimal
solution of $u_t+H(x,t,Du)=0$ in $\Man{N}$, complemented with the Ishii
conditions on $\Gamma=\Man{N-1}\cup\cdots\cup\Man{0}$. This solution is denoted
by $\mathbb{U}^-$ in \cite{BBC1,BBC2}.

The maximal solution, $\mathbb{U}^+$, was identified in \cite{BBC1,BBC2} but
only in the specific case of a $(N-1)$-dimensional discontinuity
set: $\Gamma=\Man{N-1}$, \textit{i.e.} $\Man{k}=\emptyset$ for any $k=0..(N-2)$.
The reason is that the identification of $\mathbb{U}^+$ through a suitable
control problem (involving only ``regular controls'') requires a reflection-type
argument on $\Gamma$. Thus, the problem is linked to the very definition of this
maximal solution and in the context of general HJB problems on stratified
domains, the methods used in \cite{BBC1,BBC2} do not seem to be adaptable
(except in special cases). This is to our point of view a very interesting
problem to identify this maximal solution in the general case (at least in a
framework as general as possible).


\end{document}